\documentclass[11pt,reqno,a4paper]{article}

\usepackage[a-2b]{pdfx}

\usepackage[numbers]{natbib}
\usepackage{epsfig}
\usepackage[title, titletoc]{appendix}
\usepackage[normalem]{ulem}
\usepackage[utf8]{inputenc}
\usepackage{cancel}
\usepackage[margin=1in]{geometry}

\usepackage{amsmath,amsthm,pdfsync,verbatim,graphicx,epstopdf,enumerate,amssymb}

\hypersetup{colorlinks=true,allcolors=blue}

\usepackage{float}
\usepackage[framemethod=tikz]{mdframed}
\usepackage{mathtools}

\usepackage{tikz}
\usepackage{enumitem}

\setlist[enumerate,1]{label=(\alph*)}

\newtheorem{theorem}{Theorem}[section]

\newtheorem{proposition}[theorem]{Proposition}

\theoremstyle{definition}
\newtheorem{definition}[theorem]{Definition}

\numberwithin{equation}{section}

\hypersetup{allcolors=blue}
\allowdisplaybreaks

\newcommand{\Beq}{\begin{equation}}
	\newcommand{\Eeq}{\end{equation}}
\newcommand{\beq}{\begin{equation*}}
	\newcommand{\eeq}{\end{equation*}}
\newcommand{\bal}{\begin{align}}
	\newcommand{\eal}{\end{align}}

\newcommand{\bp}{\begin{prob}}
	\newcommand{\ep}{\end{prob}}
\newcommand{\bpr}{\begin{proof}}
	\newcommand{\epr}{\end{proof}}

\newcommand{\bel}[1]{\begin{equation}\label{#1}}
	\newcommand{\ee}{\end{equation}}

\newcommand{\R}{\mathbb{R}}

\newcommand{\Lc}{\mathcal{L}}

\newcommand{\Zc}{\mathcal{Z}}

\newcommand{\Ec}{\mathcal{E}}
\newcommand{\Fc}{\mathcal{F}}

\newcommand{\Pc}{\mathcal{P}}

\newcommand{\Sb}{\mathbb{S}}

\newcommand{\Zb}{\mathbb{Z}}

\def\pa{{\partial}}
\def\Bbar{{\overline{B}}}
\def\gammadot{{\dot{\gamma}}}
\def\WF{{\text{WF}}}
\def\supp{{\text{supp }}}
\def\la{{\langle}}
\def\ra{{\rangle}}
\def\Bpone{{ \overline{B_p(|p|+1)}}}
\def\cleq{{\, \, \preceq \, \,}}
\def\Stwo{\mathbb{S}^2}
\def\ep{{\epsilon}}

\def\vspc{{\vspace{-0.2in}}}

\begin{document}

\title{Rigidity for formally determined point source scattering for an anisotropic acoustic medium}

%\date{September 4, 2026}

\author{Pradipta Chatterjee\thanks{\begin{tabular}[t]{@{}l@{}}
Centre for Applicable Mathematics, Tata Institute of Fundamental Research, Bangalore, India.\\
Email: pradipta22@tifrbng.res.in
\end{tabular}}, 
Rakesh\thanks{\begin{tabular}[t]{@{}l@{}}Department of Mathematical Sciences, University of Delaware, Newark DE, USA.
\\
Email: rakesh@udel.edu
\end{tabular}}, 
Abhilash Tushir\thanks{\begin{tabular}[t]{@{}l@{}}
Centre for Applicable Mathematics, Tata Institute of Fundamental Research, Bangalore, India.\\
Email: abhilash2296@gmail.com
\end{tabular}}
}
\date{}
\maketitle

\begingroup
\renewcommand{\thefootnote}{}
\footnotetext{\begin{tabular}[t]{@{}l@{}}\textbf{Keywords:} Inverse scattering; Fixed angle scattering; Wave equation; Riemannian metric; Eikonal equation.\end{tabular}}
\endgroup
\begin{abstract}
Consider an acoustic medium, occupying three dimensional space,
whose anisotropic sound speed is modeled by a Riemannian metric which is Euclidean outside a compact region $D$.
The medium is probed by acoustic waves generated by a finite number (dimension dependent) of point sources
outside $D$ and the medium response is measured on the boundary of $D$ for a finite time period. We show
that if these measurements agree with that of an Euclidean metric then the metric is Euclidean up to a diffeomorphism.
This result is similar to the result in \cite{ORS26a} where they used plane waves instead of point 
sources. Using point sources instead of plane waves presents new difficulties, dealt with in this article. 
\end{abstract}
\tableofcontents
\section{Introduction}

Our problem is motivated by the desire to recover the acoustic properties of a medium of dimension two or higher, 
which is inhomogeneous only on a bounded set, from the medium response, measured in the homogeneous part, to sources located in the homogeneous part. So the sources and the measurements are both in the homogeneous part. Such problems have received considerable attention over the last 40 years but most of the theoretical work on such problems has been for overdetermined problems, typically involving the use of an infinite number of sources, 
where the parameter count of the data exceeds the dimension of the space.  Our goal is to work 
on formally determined problems where the number of sources is a dimension dependent finite number, with the number 
of sources being one or two if the acoustic property is modeled by a single function.

The most significant first results and methods for formally determined inverse problems for hyperbolic PDEs, 
in dimensions two or higher, were \cite{Bukhgeum81} and \cite{IY2001}. However, these groundbreaking results and 
methods, for bounded domains, required the source to be in the inhomogeneous part of the medium and be positive 
at every point in the medium. These results are not suitable for applications where the source is in the homogeneous 
part of the medium or on the boundary of the medium (if the medium is a bounded domain). 
These methods were adapted to plane wave sources coming from infinity or point sources in the homogeneous part of 
the medium in \cite{Salo20A}, \cite{Salo20}, \cite{Salo21}, \cite{Salo22}, \cite{Salo23},
and \cite{KRS21}, \cite{KRS23} when the coefficients depend on space and time.
However, these results pertained to the determination of the zeroth order or the first order coefficients of a second order hyperbolic PDE, assuming the medium velocity (the highest order coefficient) was known. 

The determination of the medium velocity using plane wave or point sources presents a much stiffer challenge. 
In \cite{ORS26a}, it was shown that a medium with anisotropic velocity 
(modeled by a Riemannian metric\footnote{There is no requirement that this metric be close to Euclidean.}) could be distinguished from the constant velocity (modeled by the Euclidean metric) medium, 
using the data corresponding to a finite number 
(dimension dependent) plane wave sources. In \cite{ORS26b}, 
this result was extended to time dependent anisotropic mediums 
(modeled by a Lorentzian metric), again using plane wave sources, where a general Lorentzian metric was compared
with the Minkowski metric. Our result is similar to the result in \cite{ORS26a} except we use point sources instead of 
plane wave sources. Later in this introduction we describe what new difficulties arise when plane wave sources are 
replaced by point sources.

The recently posted \cite{ORS26c} generalizes substantially the work in \cite{ORS26a} by comparing two general 
Riemannian metrics (with some constraints on one of them), again using plane waves. The recently posted
\cite{ORS26d} generalizes substantially the work in \cite{ORS26b} by comparing two general Lorentzian metrics (with 
some constraints on one of them), again using plane waves. We also recommend reading the surprising results in 
\cite{AK25} for problems similar to ours, where the goal is to recover not just the Riemannian metric but the unknown 
source as well (with some conditions on the source) from partial spectral data.

We do not give a broad review of the literature for the problems being considered in this article as the
recently published \cite{ORS26a}, \cite{ORS26b} do a thorough job of it, updated in 
 \cite{ORS26c}, \cite{ORS26d}.

Let $B\subseteq \mathbb{R}^3$ denote the origin centered open unit ball. We say the differential operator
\[
\Lc_{A,b} := \pa_t^2 - \sum_{i,j=1}^3 a_{ij}(x) \pa_i \pa_j - \sum_{i=1}^3 b_i(x) \pa_i 
\]
is {\em admissible} if $a_{ij}(x), b_i(x)$ are smooth real valued functions on $\R^3$ with $A(x) = (a_{ij}(x))$  a positive 
definite symmetric matrix and $A(x) = I$, $b(x)=0$ for $x$ outside $B$.
A special case of $\Lc_{A,b}$ is when $g(x) = (g_{ij}(x))$ is a Riemannian metric on $\R^3$ with $g=g_{Eucl}$ outside $B$; we define
\[
\Box_g := \pa_t^2 - \Delta_g,
\]
where
\[
\Delta_g := |g|^{-1/2} \pa_i \left ( |g|^{1/2} g^{ij} \pa_j \right ) 
= g^{ij} \pa_i \pa_j + |g|^{-1/2} \pa_i \left ( |g|^{1/2} g^{ij} \right ) \pa_j.
\]
with $|g| = \det ( g_{ij})$ and $(g^{ij})$ being the inverse of $(g_{ij})$. 

We will also say $A$ or $g$ is {\em admissible} if $A(x)$ is smooth and positive definite on $\R^3$, $g(x)$ is a Riemannian metric on $\R^3$, and $A(x)-I$, $g(x)-I$ are supported in $B$.

Let $p$ be a point outside $\Bbar$ and let $U_p(x,t)$ be the unique distributional solution (see Proposition 
\ref{prop:forward}) of the IVP
\begin{subequations}
\begin{align}
\Lc U_p =4 \pi H(t) \, \delta(x-p), & \qquad \text{on } \R^3 \times \R
\label{eq:Upde}
\\
U_p = 0, & \qquad \text{on } \R^3 \times (-\infty, 0).
\label{eq:Upic}
\end{align}
\end{subequations}
From Proposition \ref{prop:forward}
we know that WF$(U_p)$ does not intersect the normal bundle of $\pa B \times \R$ hence, from 
\cite[Theorem 8.2.4]{H1}, $U_p$ has a trace on $\pa B \times \R$. Our long term goal is to prove the injectivity of the map
\[
\Fc_p : (A,b) \to U_p|_{\pa B \times (-\infty, T)},
\]
for some large $T>0$.  One should not be able to recover multiple functions $a_{ij}, b_i$ from the data generated by 
a single source. So a reasonable problem would be to study the injectivity of
\[
\Fc : (A,b) \to \left [ \Fc_p (A,b) \right ]_{p \in \Pc},
\]
where $\Pc$ is some finite set of points outside $\Bbar$.

On occasion, we emphasize the dependence of $U_p$ on $A,b$. On these occasions, instead
of $U_p$, we write $U_{p,A,b}$ for a general $\Lc_{A,b}$, write $U_{p,g}$ when $\Lc = \Box_g$, write
 $U_{p,A}$ when $b=0$, or write $U_{p,\rho}$ when $A = \rho I$ for some function $\rho$ and $b=0$.

We focus on three four problems. Fix a $T>0$.
\vspc
\begin{enumerate}
\item \textbf{The $\rho$-problem: } Fix a $p$ outside $\Bbar$, take $b=0$ and $A=\rho I$ for some function $\rho$. Prove the injectivity of the map 
\[
\Fc: \rho \to  U_{p,\rho}|_{\pa B \times (-\infty, T)}.
\]
\item \textbf{The $g$-problem:} Fix $\Pc$ a finite set of points outside $B$ and take $\Lc = \Box_g$ for a Riemannian metric $g$. Prove the injectivity of the map
\[
\Fc : g \to [\Fc_p(g)]_{p \in \Pc}, \qquad \text{where } \Fc_p : g \to U_{p,g}|_{\pa B \times (-\infty, T)}.
\]
One can see that if $\Psi: \R^3 \to \R^3$ is a diffeomorphism with $\Psi(x)=x$ outside $B$, then $U_{p,g}$ and $U_{p,\Psi^*(g)}$ have the same trace on $\pa B \times (-\infty, T)$. Hence the best we can hope to do is recover $g$ up to a diffeomorphism which is the identity outside $B$. 
\item \textbf{The $A$-problem:} Fix $\Pc$ a finite set of points outside $B$, take $b=0$. Prove the injectivity of the map
\[
\Fc : A \to [\Fc_p(A)]_{p \in \Pc}, \qquad \text{where } \Fc_p : A \to U_{p,A}|_{\pa B \times (-\infty, T)}.
\]
Compare this with the $g$ problem, where there is an ambiguity in the recovery of $g$.
\end{enumerate}

One may verify that these three  problems are formally determined problems. A first step towards our goal
is comparing the data for an inhomogeneous medium $\Lc_{A,b}$ with that 
for the homogeneous medium $\Lc = \Box$ - the standard wave operator. We state our results for the three problems.
Here $e_{i}$ represents the point on the positive $x_{i}$-axis, one unit from the origin and, for {\em admissible}
$\rho, A, g$ we define
\begin{gather*}
\rho_{max} := \sup_{x \in \R^3} \rho^{-1}(x),
\qquad
A_{max} := \sup_{x, v \in \R^3, ~ |v|=1}  v^T A(x)^{-1} v,
\\
g_{min} := \inf_{x, v \in \R^3, ~ |v|=1} v^T g(x) v,
\qquad
g_{max} := \sup_{x, v \in \R^3, ~ |v|=1} v^T g(x) v.
\end{gather*}

\begin{theorem}[The $\rho$ problem]\label{Roproblem} 
Suppose $\Lc_{A,b}$ is admissible with $A=\rho I$, $b=0$,
$p \in \R^3 \setminus \Bbar$, and $T>\sqrt{\rho_{\max}}(|p|+3)$. If
\[
U_{p,\rho} = U_{p,1}, \qquad \text{on } \pa B \times (-\infty, T),
\]
then $\rho=1$.
\end{theorem}

 \begin{theorem}[The $g$ problem]\label{gproblem} 
 Suppose $\Lc = \Box_g$ is admissible, $T>\sqrt{g_{\max}}(2 \sqrt{2}+3)$, and let
 \[
\Pc := \{ 2e_i\}_{i=1}^3 \cup \{ 2(e_j + e_k) : j < k \}_{j,k=1}^3
\cup \{ e_1 + 2e_2, e_2 + 2e_3, e_3 + 2e_1, e_1+e_2+e_3 \}.
\]
If
 \[
U_{p,g} = U_{p,Eucl}, \qquad \text{on } \pa B \times (-\infty, T), \qquad \forall p \in \Pc,
\]
then there is a diffeomorphism $\psi: \R^3 \to \R^3$ with $\Psi=$id outside $B$ such that $\Psi^*(Eucl)= g$.
\end{theorem}

\begin{theorem}[The $A$ problem]\label{Aproblem} 
Suppose $\Lc_{A,b}$ is admissible with $b=0$, $T>\sqrt{A_{\max}}(2 \sqrt{2}+3)$
and let
\[
\Pc := \{ 2e_i\}_{i=1}^3 \cup \{ 2(e_j + e_k) : j < k \}_{j,k=1}^3
\cup \{ e_1 + 2e_2, e_2 + 2e_3, e_3 + 2e_1, e_1+e_2+e_3 \}.
\]
If
\[
U_{p,A} = U_{p,I}, \qquad \text{on } \pa B \times (-\infty, T), \qquad \forall p \in \Pc,
\]
then $A=I$.
\end{theorem}

The $g$ problem required 13 sources and one may wonder if $g = \rho\, g_{Eucl}$ for some smooth function $\rho$, 
would one need fewer sources. 
\begin{theorem}[The $g=\rho \, g_{Eucl}$ problem]\label{gscalar}
Suppose $\Lc = \Box_g$ is admissible with $g=\rho \, g_{Eucl}$, $p \in \R^3 \setminus \Bbar$,
and $T> \sqrt{\rho_{max}}(|p|+3)$. If
\[
U_{p,g} = U_{p,Eucl}, \qquad \text{on } \pa B \times (-\infty, T),
\]
then $\rho =1$.
\end{theorem}

Theorems  \ref{Roproblem}, \ref{gproblem}, \ref{Aproblem}
are motivated by similar results in \cite{ORS26a} where plane wave sources 
were used instead of the point sources in our case. Our proofs have the same overall plan as in 
\cite{ORS26a}, but significant new difficulties arise in the point source case. Our proofs and the proofs in
\cite{ORS26a} consist of two parts:
\vspc
\begin{itemize}
\item A geometrical result such as our Proposition \ref{prop:EIC}. Our proof is quite similar to the proof of the corresponding result in \cite{ORS26a}.
\item The construction of the diffeomorphism required in Theorem \ref{gproblem} and showing that this diffeomorphism 
is the identity for Theorem \ref{Aproblem}. Using point sources instead of plane waves presents new difficulties for 
this part. For a plane wave source, as in \cite{ORS26a}, the relevant solution of the PDE in the Euclidean case is 
$H(t-x \cdot \omega)$ for some unit vector $\omega$. The proof in \cite{ORS26a} relied on the observation that $\Lc_{A,b} 1=0$
and $x \cdot \omega$ is linear in $\omega$. For the point source case, the relevant solution of the PDE in the Euclidean case is $ |x-p|^{-1} H(t - |x-p|)$, and neither $\Lc_{A,b} |x-p|^{-1} =0$ nor is $|x-p|$ linear in $p$. The new ideas needed
are in the statements and the proofs of Propositions \ref{prop:UU0}, \ref{prop:alphaprop}, \ref{prop:Vprop}.
\end{itemize}
\vspc
The plane wave source version of Theorem \ref{gscalar} was not considered in \cite{ORS26a}. The result for the plane wave can be proved by an argument similar to, but simpler than, the proof of Theorem \ref{gscalar}.

The article is organized in the following fashion. Section \ref{sec:ray} contains the statement of the well-posedness of the 
forward problem for the PDE and the statement of the geometrical result needed for the inverse problem. The proofs of 
these propositions are in the appendix as these proofs are similar to the proofs for the plane wave case in 
\cite{ORS26a}. Section \ref{sec:main} contains the proofs of the theorems, but these theorems (except Theorem
\ref{gscalar} requires an additional idea) are quick consequences 
of Propositions \ref{prop:UU0}, \ref{prop:alphaprop}, \ref{prop:Vprop} which are the crucial results and which are different 
from what is in \cite{ORS26a}. The appendix contains the proofs of the two propositions stated in Section \ref{sec:ray} 
and a simple energy estimate needed in Section \ref{sec:main}.

\noindent
{\bf Acknowledgements.} 
Pradipta Chatterjee acknowledges the support of the
Department of Atomic Energy, Government of India, for the PhD Fellowship.
Rakesh's work was was partly funded by the grant DMS 2307800 from
the National Science Foundation of USA. Abhilash Tushir acknowledges the support of the
Anusandhan National Research Foundation (ANRF), under the scheme National Post-Doctoral
Fellowship, file no.~PDF/2025/005843. 

We used the help of ChatGPT in the proof of Theorem \ref{gscalar}. We describe this help in detail at the beginning of the proof of Theorem \ref{gscalar}; see Subsection \ref{proofthm 1.4}.

%%%%%%%%%%%%%%%%
%%%%%%%%%%%%%%%%%%%

\section{The ray geometry and the forward problem}\label{sec:ray}
For $p \in \R^3$, let $B_p(r)$ denote the $p$-centered ball of radius $r$, and $B^*_p(r) = B_p(r) \setminus \{p\}$.
For a Riemannian metric $g$ on $\R^3$, $\| \cdot \|$ denotes the norm under $g$ and, for a function $f$ on $\R^3$, 
$\nabla_g f$ denotes the vector field corresponding to $df$ under the metric $g$. Also, $|\cdot|$ denotes
the Euclidean norm.

For the admissible operator $\Lc_{A,b}$, the principal symbol is 
\[
f(x,t;\xi,\tau) := -\tau^2 + a_{ij}(x) \xi_i \xi_j = -\tau^2 + g^{ij}(x) \xi_{i} \xi_j,
\qquad (x,t;\xi,\tau) \in T^*(\R^3 \times \R),
\]
where $g(x) = A(x)^{-1}$ is the Riemannian metric we associate with $\Lc_{A,b}$. The null bicharacteristics of 
$\Lc_{A,b}$, parameterized with $t$, are solutions $(x(t), t; \xi(t), \tau(t))$ of the IVP 
\begin{align*}
\dot{x}=\frac{1}{\tau}g^{-1}\xi, \qquad
\dot{\xi_{k}}=\frac{1}{2\tau}(g^{-1}\xi)^{T}\partial_{x_{k}}(g)(g^{-1}\xi), 
\qquad \dot{\tau}=0,
& \qquad k=1,2,3\\
x(t=0)=x_0, \qquad \xi(t=0)= \lambda \xi_0, \qquad \tau(0) = -\lambda, 
\end{align*}
for some $(x_0,0; \lambda \xi_0, -\lambda) \in T^*(\R^3 \times \R)$ with
\[
\xi_0^T g(x_0)^{-1} \xi_0 =1, \qquad \lambda \neq 0. 
\]
Since $g$ is Euclidean outside $B$, one can show that these null bicharacteristics are defined for all 
$t \in \R$, $x(t)$ is independent of $\lambda$, and $\xi(t)$ is homogeneous of degree $1$ in $\lambda$. 

Keeping in mind the identification of $T^*(\R^3 \times \R)$ with $T(\R^3 \times \R)$ through the Lorentzian metric
\[
h(x,t) = \begin{bmatrix} -1 & 0 \\ 0 & g(x) \end{bmatrix},
\qquad (x,t) \in \R^3 \times \R,
\]
one can show (see \cite[Subsection 2.2.1]{ORS26a}) that the null bicharacteristics for $\Lc_{A,b}$
correspond to the {\bf unit speed} geodesics $\gamma$, of the Riemannian metric $g$, with the initial condition
\[
\gamma(0) = x_0, \qquad \gammadot(0) = g(x_0)^{-1} \xi_0;
\]
the relationship is
\[
x(t) = \gamma(t), \qquad \xi(t) = \lambda \, g( \gamma(t)) \, \gammadot(t), \qquad \tau = - \lambda.
\]

For a fixed $p \in \R^3 \setminus \Bbar$, the null bicharacteristics of interest are those corresponding to the 
initial condition
\[
x(0)=p, ~ \xi(t=0) = \lambda \omega, ~ \tau(0) = - \lambda
\]
for arbitrary $\omega \in \Stwo$, $0 \neq \lambda \in \R$, and we only need the half of these null bicharacteristics 
corresponding to $t \geq0$. These half-null bicharacteristics for the $\omega, \lambda$ will be called 
the {\bf $p$-null bicharacteristics}. 

Given $p \in \R^3 \setminus \Bbar$,  a unit vector $\omega$ in $\R^3$, let $\gamma_{p,\omega}:[0,\infty) \to \R^3$
denote the (unit speed) half-geodesic, corresponding to $t\geq 0$, with 
\[
\gamma_{p,\omega}(0)=p, \qquad \gammadot_{p,\omega}(0) = \omega.
\]
The geodesics $\gamma_{p,\omega}$, collectively, will be called {\bf $p$-geodesics}. 

Consider the submanifold of $T(\R^{3+1})$ traced out by the tangent vectors of the $p$-geodesics
\[
\Lambda_p := \{ [\gamma_{p,\omega}(t), t; \; \lambda \gammadot_{p,\omega}(t), \lambda ] : \omega \in \Stwo, ~ t>0, ~ 
\lambda\in \R, ~ \lambda\neq 0 \},
\]
and its dual submanifold $\Lambda_p^*$ of $T^*(\R^{3+1})$, under the Lorentzian metric 
\[
h(x,t) = \begin{bmatrix} -1 & 0 \\ 0 & g(x) \end{bmatrix}, \qquad (x,t) \in \R^{3+1},
\]
given by
\[
\Lambda_p^* := \{  [\gamma_{p,\omega}(t), t; \; \lambda \gammadot_{p,\omega}(t)^\flat, -\lambda ] : \omega \in \Stwo, ~ t>0, ~ 
\lambda\in \R, ~ \lambda\neq 0 \},
\]
where $v^\flat$ is the dual of the tangent vector $v$ under the Riemannian metric $g$.
Then $\Lambda_p^*$ is the submanifold traced out by the $p$-null bicharacteristics.

A curve $\gamma:[a,b] \to \R^3$ will be said to be {\em admissible} if it is continuous and piecewise smooth. The length 
of an admissible curve $\gamma$ is defined as
\[
L(\gamma) := \int_a^b \| \gammadot(t)\| \, dt,
\]
where $\| \gammadot(t)\|$ denotes the length of $\gammadot(t)$ with the metric $g$. Note that $L(\gamma)$ is 
independent of the parametrization chosen for $\gamma$. For any $p,q\in\mathbb{R}^{3}$, the Riemannian distance
between $p,q$ is defined as
\[
d(p,q):=\inf\{L(\gamma): \text{ $\gamma$ an admissible curve joining $p$ to $q$}\}.
\]
One can show that $(\mathbb{R}^3,d)$ is a metric space, and $d$ is equivalent  to Euclidean metric because
\begin{align}\label{equimetric}
    \sqrt{g_{\min}} \, |p-q|\leq d(p,q)\leq \sqrt{g_{\max}} \, |p-q|.
\end{align}

Given $p \in \R^3 \setminus \Bbar$, define the time of first arrival function $\alpha_p : \R^3 \to \R$ with
\[
\alpha_p(x) = d(p,x), \qquad x \in \R^3.
\]
We see that $\alpha_p$ is Lipschitz continuous on $\R^3$ and, since $g$ is Euclidean outside $B$, we have
\[
\alpha_p(x) = |x-p|, \qquad \text{if } |x-p| < |p|-1.
\]
Define the map $\Phi_p : [0,\infty) \times \Stwo \to \R^3$ with
\[
\Phi_p(s,\omega) = \gamma_{p,\omega}(s), \qquad s \geq0, ~ |\omega|=1.
\]

The following defines a desirable property, outside $B$, of the $p$-geodesics. For example, this occurs when $g$ is the Euclidean metric.
\begin{definition}[Exterior Injectivity Condition (EIC)] 
Suppose $(\mathbb{R}^3,g)$ is an admissible Riemannian metric, $p \in \R^3 \setminus \Bbar$, 
and $T>\sqrt{g_{max}} \,(|p|+1)$. We say $(g,p, T)$ satisfies the \textit{Exterior Injectivity Condition} if the following holds: if there are $p$-geodesics $\gamma_1, \gamma_2$ and $t_1, t_2 \in (0,T)$ so that
$\gamma_1(t_1) = \gamma_2(t_2) \in \R^3 \setminus B$, then $\gamma_1 = \gamma_2$.
\end{definition}
\vspc
\noindent
Given $T > \sqrt{g_{max}}(|p|+1)$, the EIC condition prohibits different $p$-geodesics intersecting outside 
$B$ at times lying in $(0,T)$.

The following important proposition asserts that if $(g,p,T)$ satisfies the 
EIC then $\Phi$ restricted to $(0,T) \times \Stwo$ is
a diffeomorphism onto its range. So a condition on the $p$-geodesics outside $B$ forces the behavior of the
$p$-geodesics inside $B$. This proposition is an analog of   \cite[Proposition 2.3]{ORS26a}
with the geodesics with initial velocity parallel to a fixed direction $\omega$  replaced, in this article,
by geodesics starting at a fixed point $p$.
%
%%%%%%%%%%%%%
\begin{proposition}\label{prop:EIC}
Suppose $(\mathbb{R}^{3},g)$ be an admissible Riemannian metric, $p \in \R^3 \setminus \Bbar$,
and $(g,p,T)$ satisfies the EIC. Define
\[
D_{p,T} := \Phi_p ( (0,T) \times \Stwo );
\]
then $D_{p,T}$ is an open subset of $\R^3$, $\Phi_p : (0,T) \times \Stwo \to D_{p,T}$ is a diffeomorphism,
and $\Bpone \subset D_{p,T}$.
Further $\alpha_p$ is a smooth function on $D_{p,T}$ with the following properties:
\vspc
\begin{itemize}
\item for $x \in D_{p,T}$,  $x = \gamma_{p,\omega}(t)$ iff $t = \alpha_p(x)$;
\item $\| \nabla_g \alpha_p \|^2 =1$ on $D_{p,T}$;
\item if $x \in D_{p,T}$ then
$\nabla_g \alpha_p(x) = \gammadot_{p,\omega}(t)$ for the unique $(t,\omega) \in (0,T) \times \Stwo$ with 
$\Phi_p(t,\omega) = x$.
\end{itemize}
\end{proposition}
\vspc
\noindent
Here, for any smooth function $f$ on $\R^3$, $\nabla_g f = df^\sharp$ is the vector field in $T(\R^3)$ identified with
$df \in T^*(\R^3)$ through the metric $g$. In coordinates $\nabla_g f = g^{-1} \nabla f$.

We now state carefully the existence, uniqueness and some properties of the solution of the IVP \eqref{eq:Upde}, \eqref{eq:Upic}. We will use the observation, obtained using standard arguments, that 
$\WF( H(t) \delta(x-p) ) = \Zc$, where
\[
\Zc := \{ (p,0; \xi, \tau) : (0,0) \neq (\xi, \tau) \in \R^{3+1} \} \cup \{ (p,t; \xi,0) : t>0, 0 \neq \xi \in \R^3\}.
\]

\begin{proposition}\label{prop:forward}
Suppose $\Lc_{A,b}$ is admissible and $p \in \R^3 \setminus \Bbar$. The IVP \eqref{eq:Upde}, \eqref{eq:Upic} has a unique distributional solution $U_p$. Further  
\vspc
\begin{enumerate}[label=(\alph*)]
\item $\supp U_p \subset \{(x,t) \in \R^3 \times \R : t \geq \alpha_p(x) \}$, 
\item $U_p \in L^2_{loc}(\R^3 \times \R)$,
\item ${\displaystyle \WF (U_p) = \Lambda_p^* \, \cup \, \Zc}$,
\item $U_p(x,t) = \dfrac{ H(t - |x-p|)}{|x-p|}$ for $(x,t) \in \R^3 \times (-\infty, |p|-1)$.
\end{enumerate}
\end{proposition}

These two propositions are the point source analogs of corresponding propositions in \cite{ORS26a} for plane wave sources. We postpone the proofs to the appendix as the proofs are similar to those in \cite{ORS26a}.

%%%%%%%%%%%%%%%%%%%%%%%%%
%%%%%%%%%%%%%%%%%%%%%%%%%%
\section{Proofs of the theorems}\label{sec:main}

We indicate the dependence of $U_p$ on $A,b$ or $g$ by writing $U_{p,A,b}$, $U_{p,g}$ and $U_{p,A}$ when $b=0$. 
We use $U_{p,Eucl}$ to represent the solution corresponding to $A=I, b=0$. We use the notation from Proposition
\ref{prop:EIC} and define the the elliptic operator
\[
\Ec := a_{ij} \pa_i \pa_j + b_i \pa_i.
\]
For convenience, for a Riemannian metric $g$ on $\R^3$, and a function $f$ on $\R^3$, $\nabla f$ will stand for 
$\nabla_gf$.

%%%%%%%%%%
\subsection{The crucial propositions}\label{sec:relations}

The following proposition is the first step in the proofs of all the theorems.

\begin{proposition}\label{prop:UU0}
Suppose $\Lc_{A,b}$ is admissible, $p \in \R^3 \setminus \Bbar$ and $T> \sqrt{A_{max}}(|p|+3)$. If
\[
U_{p,A,b} = U_{p,Eucl} \qquad \text{on } \pa B \times (-\infty, T)
\]
then there is a function $\alpha_p$ on $\R^3$ such that the following holds (here $g=A^{-1}$):
\vspc
\begin{enumerate}[label=(\alph*)]
\item\label{p3.1itema} $\alpha_p$ is smooth and never zero on $\R^3 \setminus \{p\}$;
\item\label{p3.1itemb} $ \| \nabla_g \alpha_p \|^2 =1$ on $\R^3 \setminus \{p\}$;
\item\label{p3.1itemc} $\Ec_{A,b} \, \alpha_p^{-1}=0$ on $\R^3 \setminus \{p\}$;
\item\label{p3.1itemd} $\alpha_p(x) = |x-p|$ for $x \in \R^3 \setminus B$.
\end{enumerate}
\end{proposition}
%
%%%%%%%%%

The next proposition claims certain relations between $\alpha_p, \alpha_q$ for different $p,q$. These relations
will be used to construct the diffeomorphism promised in Theorem \ref{gproblem} and this diffeomorphism
turns out to be the identity under the conditions of Theorem \ref{Aproblem}. 
Here the inner product is the one associated with $g=A^{-1}$ and $\nabla$ represents $\nabla_g$.

\begin{proposition}\label{prop:alphaprop}
Suppose $\Lc_{A,b}$ is admissible and we take $g=A^{-1}$. Suppose $k \in \R^3 \setminus \Bbar$ and
\vspc
\begin{enumerate}[label=(\alph*)]
\item\label{p3.2itema} $\alpha_k$ is smooth on $\R^3 \setminus \{k\}$
\item\label{p3.2itemb} $\|\nabla \alpha_k\|^2=1$ on $\R^3 \setminus \{k\}$
\item\label{p3.2itemc} $\Ec_{A,b} \alpha_k^{-1} =0$ on $\R^3 \setminus \{k\}$
\item\label{p3.2itemd} $\alpha_k(x) = |x-k|$ for $x \in \R^3 \setminus B$.
\end{enumerate}
\vspc
Then the following holds on $\R^3$:
\vspc
\begin{enumerate}[label=(\roman*)]
\item\label{p3.2itemi}  $\Ec_{A,b} (\alpha_k^2)=6$, ~~ $\displaystyle \Ec_{A,b} (\|\nabla\alpha_k^2\|^2)=24$.
\item\label{p3.2itemii}  If $p,q, (p+q)/2\in \R^3 \setminus \Bbar$ and \ref{p3.2itema}-\ref{p3.2itemd} hold for $k=p,q,(p+q)/2$, then
\[
\Ec_{A,b} \langle\nabla\alpha_p^2,\nabla\alpha_q^2\rangle=24.
\] 
\end{enumerate}
\end{proposition}
\vspc

The next proposition constructs the diffeomorphism promised in Theorem \ref{gproblem} and it also plays a role
in the proof of Theorem \ref{Aproblem}.
\begin{proposition}\label{prop:Vprop} Suppose $\Lc_{A,b}$ is admissible, $g=A^{-1}$, and 
$\alpha_p$ has the properties \ref{p3.1itema}-\ref{p3.1itemd} of Proposition \ref{prop:UU0}, for each $p$ in the set
\[
\Pc = \{ 2e_i\}_{i=1}^3 \cup \{ 2(e_j + e_k) : j < k \}_{j,k=1}^3
\cup \{ e_1 + 2e_2, e_2 + 2e_3, e_3 + 2e_1, e_1+e_2+e_3 \}.
\]
For $i,j,k =1,2,3$ and $j \neq k$, define the smooth function
\begin{align}
V^i_{j,k} := \frac{1}{4} \left ( \alpha_{2 e_i}^2 - \alpha_{2(e_j + e_k)}^2 - |2 e_i|^2 + |2(e_j+e_k)|^2 \right ),
\qquad \text{on } \R^3
\label{eq:Vijkdef}
\end{align}
and the map $\Psi : \R^3 \to \R^3$ with
\[
\Psi(x) = \left ( V_{1,2}^2(x), V_{2,3}^3(x), V_{3,1}^1(x) \right ), \qquad x \in \R^3.
\]
Then $\Psi$ is a diffeomorphism with $\Psi=$id outside $B$ and $\Psi^*(Eucl)=g$.
\end{proposition}

%%%%%%%%%%%%%%%%%%
%%%%%%%%%%%%%%%%%%%
\subsection{Proof of Theorem \ref{Roproblem}}

Here $A=\rho I$, $b=0$, $g=\rho^{-1} I$ and $\Ec_{A,b} = \rho \Delta$. From Proposition \ref{prop:UU0},
\[
\Delta \alpha_p^{-1} = 0 \qquad \text{on } \R^3 \setminus \{p \}
\]
and
\[
\alpha_p(x) = |x-p|, \qquad x \in \R^3 \setminus B.
\]
Since $\Delta |x-p|^{-1} =0$ on $\R^3 \setminus \{p\}$, from unique continuation for 
second order elliptic operators, we have $\alpha_p(x) = |x-p|$ on $\R^3$. Hence on $\Bbar$, we have
\[
1 = \| \nabla_g \alpha_p \|^2 = \rho(x) |\nabla \alpha_p(x)|^2 = \rho(x) | (x-p)/|x-p|| = \rho(x).
\]

%%%%%%%%%%%%%%%%%%%%%%%%%%%%
\subsection{Proof of Theorem \ref{gproblem}}

Here $A,b$ are such that $\Ec_{A,b} = \Delta_g$.
From Proposition \ref{prop:UU0}, we know that for each $p \in \Pc$, $\alpha_p$ has the properties \ref{p3.1itema}-\ref{p3.1itemd} in 
Proposition \ref{prop:UU0}.  Hence the theorem follows from Proposition \ref{prop:Vprop}.

%%%%%%%%%%%%%%%%%%%
\subsection{Proof of Theorem \ref{Aproblem}}

Here $b=0$ and $g=A^{-1}$. 

From Proposition \ref{prop:UU0}, we know that for each $p \in \Pc$, $\alpha_p$ has properties \ref{p3.2itema}-\ref{p3.2itemd} listed in 
Proposition \ref{prop:UU0}. So, from Proposition \ref{prop:Vprop}, we know that $g = \Psi^{*}(Eucl)$.

Now, from the definition of $V^i_{j,k}$ and \ref{p3.2itemi} of Proposition \ref{prop:alphaprop}, we have
\[
\Lc_{A} V^i_{j,k}  =0, \qquad \text{on } \R^3.
\]
Also, from the definition of $V^i_{j,k}$ we know that
\[
V^i_{j,k}(x) = x \cdot(-e_i + e_j + e_k), \qquad \text{on } \R^3 \setminus B.
\]
Noting that $  \Lc_A (x \cdot(-e_i + e_j + e_k) ) =0$ on $\R^3$, unique continuation for second elliptic 
operators implies
\[
V^i_{j,k}(x) = x \cdot(-e_i + e_j + e_k), \qquad \text{on } \R^3.
\]
Hence
\[
\Psi(x) = ( x \cdot e_1, x \cdot e_2, x \cdot e_3) = x, \qquad x \in \R^3.
\]
Therefore $g = Eucl$.
%%%%%
%%%%%%%%%%%%

\subsection{Proof of Theorem \ref{gscalar}}\label{proofthm 1.4} 

We used ChatGPT's help in this proof. We reduced the problem to a unique continuation question for a nonlinear PDE
and we asked ChatGPT to find results in the literature addressing this unique continuation question. ChatGPT showed that a change of the dependent variable reduces the question to one answered by a 
well known theorem in \cite{Evans-pde-book}.

Here $g=\rho I$ with $\rho=1$ outside $B$.
From Proposition \ref{prop:UU0}, there is a smooth function
$\alpha$ on $\R^3\setminus \{p\}$ with $\alpha(x) = |x-p|$ on $\R^3 \setminus B$ and (using the summation notation)
\[
|\nabla \alpha|^2 = \rho, \qquad \pa_i ( \rho\, \pa_i ( \alpha^{-1})) =0, \qquad \text{on } \R^3 \setminus \{p\}.
\]
Hence $\nabla \alpha$ is never zero on $\R^3 \setminus \{p\}$ and
\begin{align}
\pa_i ( |\nabla \alpha|^2 \pa_i (\alpha^{-1})) = 0,   \qquad \text{on } \R^3 \setminus \{p\}.
\label{eq:onlyalpha}
\end{align}
For future use, one observes that  $\alpha(x) = |x-p|$ satisfies \eqref{eq:onlyalpha}; this is also clear from the fact
that $\alpha(x) = |x-p|$ when $\rho \equiv 1$. 

Below we use the terminology from \cite[Section 9.1]{Evans-pde-book}.
Now $\alpha>0$ on $\R^3\setminus \{p\}$, so $\beta=\alpha^{1/3}$ is a smooth function on $\R^3\setminus\{p\}$.
One can verify that \eqref{eq:onlyalpha} implies
\begin{align}
\nabla \cdot (|\nabla \beta|^2 \, \nabla \beta) =0, \qquad \text{on } \R^3 \setminus \{p\},
\label{eq:betaonly}
\end{align}
and $\beta(x) = |x-p|^{1/3}$ on $\R^3$. Further, $|x-p|^{1/3}$ is another solution of \eqref{eq:betaonly}. Now the vector field
\[
F(\xi) = |\xi|^2 \xi, \qquad \xi \in \R^3
\]
is monotone because
\[
F'(\xi) = |\xi|^2 I + 2 \xi \otimes \xi
\]
is positive definite. Hence, by \cite[Section 9.1, Theorem 4]{Evans-pde-book}, 
\[
\beta(x) = |x-p|^{1/3} \qquad \text{on } \R^3
\]
implying $\alpha(x) = |x-p|$ on $\R^3$. Hence
\[
\rho = |\nabla \alpha|^2 = 1 \qquad \text{on } \R^3 \setminus \{p\}.
\]

%%%%%%%%%%%%%%
%%%%%%%%%%%

\subsection{Proof of Proposition \ref{prop:UU0}}

Since $p$ is fixed, we do not show the dependence on $p$ in the notation, so $U_p, U_{p,g_{Eucl}}, \alpha_p$ 
will be written as $U, U_{Eucl}, \alpha$. 

Since $T> \sqrt{g_{max}} (|p|+1)$, as argued in the proof of Proposition \ref{prop:EIC}, we have $\Bpone \subset
D_T$.

We are given that $U=U_{Eucl}$ on $\pa B \times (-\infty, T)$. 
By imitating the proof of a similar claim in the proof of \cite[Proposition 3.1]{ORS26a}, one can show that
\begin{equation}
U=U_{Eucl} = \frac{H(t-|x-p|)}{|x-p|}, \qquad \text{on } (\R^3 \setminus B) \times (-\infty, T).
\label{eq:UH}
\end{equation}

The equation \eqref{eq:UH} implies $\WF(U) = \WF(U_{Eucl})$ on $(\R^3 \setminus B) \times (-\infty, T)$, 
hence $\Lambda_{p,g} = \Lambda_{p,Eucl}$ when restricted to $(\R^3 \setminus B) \times (-\infty, T)$. 
Hence the following holds:
\vspc
\begin{itemize}
\item $D_T  = B_p(T)$,
\item For any $p$-geodesic $\gamma$, the part of $\gamma|_{(0,T)}$ lying in $\R^3 \setminus \Bbar$ must be a part of 
a half line originating at $p$,
\item If $\gamma$ is a $p$-geodesic and $\gamma(t) \in D_T \setminus B$ for some $t \in (0,T)$ then
$L(\gamma|_{[0,t]}) = |\gamma(t)-p|$.
\end{itemize}
\vspc
Since no two half-lines through $p$ can intersect and $T>\sqrt{g_{max}}(|p|+1)$, we see that $(g,p,T)$ satisfy the EIC.
Therefore Proposition \ref{prop:EIC} implies that
$\Phi : (0,T) \times \Stwo \to B_p^*(T)$ is a diffeomorphism and $\alpha(x)$ is a smooth function on $B_p^*(T)$.
Further, from the third bullet above,
\begin{align}
\alpha(x) = d(p,x) = |x-p|, \qquad \text{for } x \in B_p(T) \setminus B.
\label{eq:alphaout}
\end{align}
The singular support of $U$ is the projection of $\Lambda_p$ onto $\R^3 \times \R$, hence
\[
\text{sing supp } U_p = \{ (\gamma_{p,\omega}(t), t) : \omega \in \Stwo, ~ t >0 \} \cup \{ (p,t) : t \geq 0 \}.
\]
For each $0<t<T$, $\gamma_{p,\omega}(t) \in B_p^*(T)$ and 
\[
\alpha( \gamma_{p,\omega}(t) ) = t,
\]
hence 
\[
\text{sing supp } U \cap \{ t<T \} = \{ (x, \alpha(x) ) : x \in B_p^*(T) \} \cup \{ (p,t) : 0 \leq t <T \}.
\]
Now $U \in L^2_{loc}(\R^3 \times \R)$ and $U$ is supported in the region $t \geq \alpha(x)$, so
\[
U(x,t) = u(x,t) H( t- \alpha(x) ), \qquad (x,t) \in \R^3 \times (-\infty, T)
\]
with $u(x,t)$ smooth in the punctured conical region 
\[
Q_+ := \{ (x,t) : 0 < \alpha(x) < t < T \}.
\]
Hence $\Lc_{A,b} U =0$ on $(\R^3 \setminus \{p\}) \times \R$ and $U=U_{Eucl}$ 
on $(\R^3 \setminus B) \times (-\infty, T)$
implies
\begin{subequations}
\begin{align}
\Lc_{A,b} u =0, & \qquad \text{on } Q_+,
\label{eq:upde}
\\
u(x,t) = \frac{1}{|x-p|}, & \qquad \text{on } Q_+ \setminus (B \times (-\infty, T)).
\label{eq:uout}
\end{align}
\end{subequations}
Differentiating these equations w.r.t $t$, we obtain
\begin{subequations}
\begin{align}
\Lc_{A,b} u_t =0, & \qquad \text{on } Q_+,
\label{eq:utde}
\\
u_t(x,t) = 0, & \qquad \text{on } Q_+ \setminus (B \times (-\infty, T)).
\label{eq:utout}
\end{align}
\end{subequations}
Define
\[
\alpha_{max} := \max_{x \in \Bbar} \alpha(x),
\]
and note that $\alpha_{max} \leq \sqrt{g_{max}}(|p|+1)$. The radius of $\Bbar$, measured from $0$, under the 
metric $g$, is at most $\sqrt{g_{max}}$ and, by hypothesis,
\[
T - \alpha_{max} \geq T - \sqrt{g_{max}} (1+|p|) > 2 \sqrt{g_{max}},
\]
so the unique continuation argument in the proof of \cite[Proposition B.1]{ORS26a} implies
\[
u_t=0, ~ u_{tt}=0, \qquad \text{on } \Bbar \times \{t=\tau\},
\]
for some $\tau < (T+ \alpha_{max})/2$. Then a standard energy estimate argument applied to 
\eqref{eq:utde}, \eqref{eq:utout} implies $u_t=0$ on $\Bbar \times [\tau,T)$. Further, an application of
Proposition \ref{prop:energy} with $\beta(x) = \alpha(x) + \ep$, for all small $\ep>0$, implies
\[
u_t = 0 \qquad \text{in } Q_+.
\]
Hence, there is a smooth function $v(x)$ on $B^*_p(T)$ such that
\begin{subequations}
\begin{align}
U(x,t) = v(x) H(t-\alpha(x)), & \qquad \text{on } (\R^3 \setminus \{p\}) \times (-\infty, T),
\label{eq:Uv}
\\
v(x) = \frac{1}{|x-p|}, & \qquad \text{on } B^*_p(T) \setminus B.
\label{eq:vout}
\end{align}
\end{subequations}
Now $\Lc_{A,b} U =0$ on $(\R^3 \setminus \{p\}) \times \R$. On the set $(\R^3 \setminus \{p\}) \times (-\infty, T)$, a standard calculation shows that
\begin{align*}
\Lc_{A,b} U(x,t) = ( 2 a_{ij} \pa_j \alpha \, \pa_i v + (\Ec_{A,b} \alpha ) v ) \, \delta (t- \alpha(x))
- (\Ec_{A,b} v ) H(t-\alpha(x)),
\end{align*}
hence
\begin{subequations}
\begin{align}
2 a_{ij} \pa_j \alpha \, \pa_i v + (\Ec_{A,b} \alpha ) v =0, & \qquad \text{on } B^*_p(T)
\label{eq:vtrans}
\\
\Ec_{A,b} v = 0, & \qquad \text{on } B^*_p(T).
\label{eq:EAbv}
\end{align}
\end{subequations}
Using \eqref{eq:vtrans} and \eqref{eq:EAbv} we have
\[
\Ec_{Ab} (\alpha v) = \alpha \Ec_{A,b}v + (\Ec_{A,b} \alpha) v + 2 a_{ij} \pa_j \alpha \pa_i v = 0,
\qquad \text{on } B^*_p(T).
\]
Also, from \eqref{eq:alphaout} and \eqref{eq:vout} we have
\[
\alpha v = 1, \qquad \text{on } B_p^*(T) \setminus B.
\]
Noting that $\Ec_{A,b} 1 =0$ on $\R^3$, from unique continuation for second order elliptic PDEs, we have
\[
\alpha v = 1 \qquad \text{on } B_p^*(T).
\]
Hence
\[
v = \alpha^{-1} \qquad \text{on } B_p^*(T),
\]
so \eqref{eq:EAbv} implies
\[
\Ec_{A,b} \alpha^{-1} = 0, \qquad \text{on } B_p^*(T).
\]
Now $\alpha(x) = |x-p|$ on $B_p(T) \setminus B$, so if we extend $\alpha(x)$ to 
$\R^3$ by taking $\alpha(x)=|x-p|$ outside
$B$ then $\alpha(x)$ is smooth on $\R^3 \setminus \{p\}$ and, since $A=I, b=0$ outside $B$ and 
$\Delta |x-p|^{-1} =0$ on $\R^3 \setminus \{p\}$ we see that \ref{p3.1itema}, \ref{p3.1itemc}, \ref{p3.1itemd} are true. Finally, we already know that
$\|\nabla_g \alpha_p \|^2=1$ on $B^*_p(T)$ and $g=I$ outside $B$ and 
$|\nabla |x-p| |^2=1$ on $\R^3 \setminus \{p\}$ we see that \ref{p3.1itemb} also holds.

%%%%%%%%%%%%%%%%%
%%%%%%%%%%%%%%%%%%%%%%%%

\subsection{Proof of Proposition \ref{prop:alphaprop}}

$\alpha_k$ is smooth on $\R^3 \setminus \{k\}$ so $\alpha_k^2$ is smooth on $\R^3$ because
$\alpha_k(x) = |x-k|$ for $x$ near $k$. 
\vspc
\begin{enumerate}[label=(\roman*)]
\item On $\R^3 \setminus \{k\}$,
\begin{align*}
\Ec_{A,b} \alpha_k^{-1} & = a_{ij} \pa_i \pa_j \alpha_k^{-1} + b_i \pa_i \alpha_k^{-1}
 =  - a_{ij} \pa_i \left ( \alpha_k^{-2} \pa_j \alpha_k \right ) - b_i \alpha_k^{-2} \pa_i \alpha_k
 \\
 & = - \alpha_k^{-2} \Ec_{A,b} \alpha_k + 2 \alpha_k^{-3} a_{ij} \pa_i \alpha_k \pa_j \alpha_k
 \\
 & = - \alpha_p^{-2} \Ec_{A,b} \alpha_k + 2 \alpha_k^{-3} \|\nabla \alpha_k\|^2.
\end{align*}
Hence \ref{p3.1itemb} and \ref{p3.1itemc} imply
\[
\Ec_{A,b} \alpha_k = \frac{2}{\alpha_k}.
\]
Therefore, on $\R^3 \setminus \{k\}$, hence on $\R^3$, we have
\begin{align*}
\Ec_{A,b} \alpha_k^2 & 
 =  2 a_{ij} \pa_i \left ( \alpha_k \pa_j \alpha_k \right ) + 2 \alpha_k b_i  \pa_i \alpha_k
  = 2\alpha_k \Ec_{A,b} \alpha_k + 2 a_{ij} \pa_i \alpha_k \pa_j \alpha_k
   = 4 + 2 = 6.
\end{align*}
We have $\nabla_g \alpha_k^2 = 2 \alpha_k \nabla_g \alpha_k$, hence
$\|\nabla \alpha_k^2\|^2 = 4 \alpha_k^2$. Therefore
\[
\Ec_{A,b} \left ( \|\nabla \alpha_k^2\|^2 \right ) = 4 \Ec_{A,b} ( \alpha_k^2) = 24.
\]

%%%%%
\item 
This will be a consequence of a sequence of relations.
Suppose $c,d,p,q \in \R^3 \setminus \Bbar$ and $\alpha_k$ has properties \ref{p3.2itema}-\ref{p3.2itemd} for  $k=c,d,p,q$.

Our first claim is that if $c-d= \pm(p-q)$ then
\begin{align}
\Ec_{A,b} \langle\nabla\alpha_{c}^2,\nabla\alpha_{d}^2\rangle=
\Ec_{A,b} \langle\nabla\alpha_{p}^2,\nabla\alpha_{q}^2\rangle,
\qquad \text{on } \R^3.
\label{eq:cmdpmq}
\end{align}
We prove this claim first when $c-d=p-q$. For $x \in \R^3 \setminus B$, we have 
\begin{align*}
\alpha_c^{2}-\alpha_d^2-\alpha_p^2+\alpha_{q}^2& = |x-c|^2 - |x-d|^2 - |x-p|^2 + |x-q|^2
\\
& = 2x\cdot(d-c+p-q)+|c|^2-|d|^2-|p|^2+|q|^2
\\
& =|c|^2-|d|^2-|p|^2+|q|^2.
\end{align*}
Define 
\[
f(x) := \alpha_{c}^{2}(x)-\alpha_{d}^2(x)-\alpha_{p}^2(x)+\alpha_{q}^2(x)-(|c|^2-|d|^2-|p|^2+|q|^2);
\]
then $f(x)=0$ on $\R^3 \setminus B$ and, using \ref{p3.2itemi}, we have $\Ec_{A,b} f =0$ on $\R^3$.
So unique continuation for second order elliptic operators 
implies $f=0$ on $\R^3$, therefore
\begin{align}
\nabla \alpha_c^2 - \nabla \alpha_d^2 = \nabla \alpha_p^2 - \nabla \alpha_q^2, \qquad \text{on } \R^3,
\qquad \text{if } c-d=p-q.
\label{eq:nablareln}
\end{align}
Hence
\[
\left \| \nabla \alpha_c^2 - \nabla \alpha_d^2 \right \|^2 = \left \| \nabla \alpha_p^2 - \nabla \alpha_q^2 \right \|^2.
\]
This results in
\[
 \| \nabla \alpha_c^2\|^2 +  \| \nabla \alpha_d^2  \|^2 
 - 2 \la \nabla \alpha_c^2, \nabla \alpha_d^2 \ra
=  \| \nabla \alpha_p^2\|^2 +  \| \nabla \alpha_q^2  \|^2 
 - 2 \la \nabla \alpha_p^2, \nabla \alpha_q^2 \ra.
 \]
 Hence applying $\Ec_{A,b}$ to the equation and using \ref{p3.2itemi}, we obtain
 \[
 \Ec_{A,b} \la \nabla \alpha_c^2, \nabla \alpha_d^2 \ra
 =  \Ec_{A,b} \la \nabla \alpha_p^2, \nabla \alpha_q^2 \ra,
 \qquad \text{on } \R^3.
 \]
 Since $c-d=-(p-q)$ implies $c-d=q-p$, and the inner product is symmetric, we have proved \eqref{eq:cmdpmq}
 in both cases.
 
 Now we claim that if $c+d=p+q$ then
 \begin{align}
  \Ec_{A,b} \langle\nabla\alpha_{c}^2,\nabla\alpha_{d}^2\rangle=
\Ec_{A,b} \langle\nabla\alpha_{p}^2,\nabla\alpha_{q}^2\rangle,
\qquad \text{on } \R^3.
\label{eq:cpdppq}
\end{align}
We prove this claim next. Now $c+d=p+q$ implies $c-p=q-d$, hence \eqref{eq:nablareln} implies
 \[
 \nabla \alpha_c^2 - \nabla \alpha_p^2 = \nabla \alpha_q^2 - \nabla \alpha_d^2
 \]
 which implies
 \begin{align}
  \nabla \alpha_c^2 =  \nabla \alpha_p^2 + \nabla \alpha_q^2 - \nabla \alpha_d^2.
  \label{eq:cpqd}
 \end{align}
We obtain two relations from \eqref{eq:cpqd}. Taking the norm, we obtain
\begin{align*}
 \|\nabla \alpha_c^2\|^2 &=    \|\nabla \alpha_p^2\|^2 +  \|\nabla \alpha_{q}^2\|^2+    \|\nabla \alpha_{d}^2\|^2
 +2\la \nabla\alpha_{p}^2,\nabla\alpha_q^2\ra
 - 2\la \nabla\alpha_q^2, \nabla \alpha_d^2\ra
 -2\la \nabla \alpha_{d}^2, \nabla \alpha_{p}^2\ra
\end{align*}
and, taking the $g$ inner product of \eqref{eq:cpqd} with $\nabla \alpha_c^2$, we obtain
\begin{align*}
    \|\nabla\alpha_c^2\|^2&= \langle \nabla\alpha_{p}^2,\nabla\alpha_c^2\rangle+\langle \nabla\alpha_{q}^2,\nabla\alpha_c^2\rangle-\langle \nabla\alpha_d^2,\nabla\alpha_c^2\rangle.
\end{align*}
Applying $\Ec_{A,b}$ to both equations and using (i), we obtain
\begin{align*}
 24&= \Ec_{A,b} \left ( - \langle \nabla\alpha_{p}^2,\nabla\alpha_{q}^2\rangle
 + \langle \nabla\alpha_{q}^2,\nabla\alpha_{d}^2\rangle
 + \langle \nabla\alpha_{d}^2,\nabla\alpha_{p}^2\rangle \right )
 \\
24&= \Ec_{A,b} \left ( \langle \nabla\alpha_{p}^2,\nabla\alpha_{c}^2\rangle
+\langle \nabla\alpha_{q}^2,\nabla\alpha_{c}^2\rangle
- \langle \nabla\alpha_{d}^2,\nabla\alpha_{c}^2\rangle \right ).
\end{align*}
Subtracting the second equation from the first and using \eqref{eq:cmdpmq}, we obtain
\begin{align*}
 0&= -\Ec_{A,b} \langle \nabla\alpha_{p}^2, \nabla\alpha_{q}^2\rangle
 +\Ec_{A,b} \langle \nabla\alpha_{d}^2,\nabla\alpha_{c}^2\rangle
   +\underbrace{\Ec_{A,b} \langle \nabla\alpha_{q}^2, \nabla\alpha_{d}^2\rangle 
 -\Ec_{A,b} \langle \nabla\alpha_{p}^2, \nabla\alpha_{c}^2\rangle}_{=0 \text{ since }q-d=-(p-c)}
 \\
 & \qquad 
 +\underbrace{\Ec_{A,b} \langle \nabla\alpha_{d}^2,\nabla\alpha_{p}^2\rangle
 -\Ec_{A,b} \langle \nabla\alpha_{q}^2,\nabla\alpha_c^2\rangle}_{=0\text{ since }d-p=q-c}.
\end{align*}
Hence
\begin{equation*}
\Ec_{A,b}
\langle\nabla\alpha_{c}^2,\nabla\alpha_{d}^2\rangle=\Ec_{A,b} \langle\nabla\alpha_{p}^2,\nabla\alpha_{q}^2\rangle,
\qquad \text{on } \R^3,
\end{equation*}
proving \eqref{eq:cpdppq}.

We now finish the proof of \ref{p3.2itemii}. Since $p+q=\frac{p+q}{2}+\frac{p+q}{2}$ and \ref{p3.2itema}-\ref{p3.2itemd} hold for
$k=p,q,(p+q)/2$, using \ref{p3.2itemi} and \eqref{eq:cpdppq}, we have
\[
 \Ec_{A,b} \langle \nabla \alpha_{p}^2, \nabla \alpha_q^2\rangle
 =
 \Ec_{A,b} \langle \nabla \alpha_{\frac{p+q}{2}}^2, \nabla \alpha_{\frac{p+q}{2}}^2\rangle
 =\Ec_{A,b}  \|\nabla \alpha_{\frac{p+q}{2}}^2\|^2=24.
\]
\end{enumerate}
\vspc
%

%%%%%%%%%%%%%%
\subsection{Proof of Proposition \ref{prop:Vprop}}

From \eqref{eq:Vijkdef} and that $\alpha_p(x)=|x-p|$ outside $B$, we have
\begin{align}
V^{i}_{j,k}(x)  =x\cdot(-e_{i}+e_{j}+e_{k}), \qquad  \text{on } \R^3 \setminus B.
\label{eq:Vij4}
\end{align}
Also, from \eqref{eq:Vijkdef} and \ref{p3.2itemi} of Proposition \ref{prop:alphaprop} we have
\begin{align}
\Ec_{A,b}  V^{i}_{j,k} &= 0,  \qquad \text{on } \R^3.
\label{eq:Vij1}
\end{align}
Also, from \eqref{eq:Vijkdef} we have
\[
\|\nabla  V^{i}_{j,k}\|^{2}  =\frac{1}{16}\left(\|\nabla\alpha_{2e_{i}}^2\|^2+\|\nabla\alpha_{2e_{j}+2e_{k}}^2\|^2-2 
\langle\nabla\alpha_{2e_{i}}^2,\nabla\alpha_{2e_{j}+2e_{k}}^2 \rangle\right),
\qquad \text{on } \R^3,
\]
so, using \ref{p3.2itemi} and \ref{p3.2itemii} from Proposition \ref{prop:alphaprop}, 
\begin{align}
\Ec_{A,b} \left(\|\nabla  V^{i}_{j,k}\|^{2}\right)& =0
\qquad \text{on } \R^3, \qquad \text{if } e_i+e_j+e_k \in \Pc.
\label{eq:Vij3}
\end{align}
From \eqref{eq:Vij4}, we see that
\[
 V^{1}_{2,3}(x)=x\cdot(-e_{1}+e_{2}+e_{3})=-  V^{2}_{1,2}(x)+  V^{3}_{2,3}(x)+  V^{1}_{3,1}(x),
 \qquad \text{outside } B,
\]
hence unique continuation and \eqref{eq:Vij1} imply that 
\[
 V^{1}_{2,3}(x) = -  V^{2}_{1,2}(x)+  V^{3}_{2,3}(x)+  V^{1}_{3,1}(x), \qquad \text{on } \R^3.
 \]
 We can obtain two other such relations in the similar manner, and we list all of them here.
\begin{subequations}
\begin{alignat*}{2}
  V^{1}_{2,3} & =-  V^{2}_{1,2}+  V^{3}_{2,3}+  V^{1}_{3,1},
  \qquad & & \text{on } \R^3
  \\
 V^{2}_{3,1} & = V^{2}_{1,2}-  V^{3}_{2,3}+  V^{1}_{3,1}, \qquad & & \text{on } \R^3
 \\
 V^{3}_{1,2} &=  V^{2}_{1,2}+  V^{3}_{2,3}-  V^{1}_{3,1}, \qquad & & \text{on } \R^3.
\end{alignat*}
\end{subequations}
Hence taking the gradient, the following relations holds on $\R^3$:
\begin{align*}
    \|\nabla   V^{1}_{2,3}\|^{2}=&\|\nabla   V^{2}_{1,2}\|^{2}+\|\nabla   V^{3}_{2,3}\|^{2}+\|\nabla    V^{1}_{3,1}\|^{2}\\&-2\langle\nabla   V^{2}_{1,2},\nabla   V^{3}_{2,3}\rangle+2\langle \nabla   V^{3}_{2,3},\nabla   V^{1}_{3,1}\rangle-2\langle\nabla   V^{2}_{1,2},\nabla   V^{1}_{3,1}\rangle;
    \\
      \|\nabla   V^{2}_{3,1}\|^{2}=&\|\nabla   V^{2}_{1,2}\|^{2}+\|\nabla   V^{3}_{2,3}\|^{2}+\|\nabla    V^{1}_{3,1}\|^{2}\\&-2\langle\nabla   V^{2}_{1,2},\nabla   V^{3}_{2,3}\rangle-2\langle \nabla   V^{3}_{2,3},\nabla   V^{1}_{3,1}\rangle+2\langle\nabla   V^{2}_{1,2},\nabla   V^{1}_{3,1}\rangle
      \\
       \|\nabla   V^{3}_{1,2}\|^{2}=&\|\nabla   V^{2}_{1,2}\|^{2}+\|\nabla   V^{3}_{2,3}\|^{2}+\|\nabla    V^{1}_{3,1}\|^{2}\\&+2\langle\nabla   V^{2}_{1,2},\nabla   V^{3}_{2,3}\rangle-2\langle \nabla   V^{3}_{2,3},\nabla   V^{1}_{3,1}\rangle-2\langle\nabla   V^{2}_{1,2},\nabla   V^{1}_{3,1}\rangle.
\end{align*}
Solving this system of equations, the following holds on $\R^3$:
\begin{align*}
    \langle\nabla   V^{2}_{1,2},\nabla   V^{3}_{2,3}\rangle&=\frac{1}{2}\left(\|\nabla   V^{2}_{1,2}\|^{2}+\|\nabla   V^{3}_{2,3}\|^{2}+\|\nabla    V^{1}_{3,1}\|^{2}\right)-\frac{1}{4}\left( \|\nabla   V^{1}_{2,3}\|^{2}+ \|\nabla   V^{2}_{3,1}\|^{2}\right)
    \\
    \langle \nabla   V^{3}_{2,3},\nabla   V^{1}_{3,1}\rangle&=\frac{1}{2}\left(\|\nabla   V^{2}_{1,2}\|^{2}+\|\nabla   V^{3}_{2,3}\|^{2}+\|\nabla    V^{1}_{3,1}\|^{2}\right)-\frac{1}{4}\left(  \|\nabla   V^{2}_{3,1}\|^{2}+ \|\nabla   V^{3}_{1,2}\|^{2}\right)
    \\
   \langle\nabla   V^{2}_{1,2},\nabla   V^{1}_{3,1}\rangle&=\frac{1}{2}\left(\|\nabla   V^{2}_{1,2}\|^{2}+\|\nabla   V^{3}_{2,3}\|^{2}+\|\nabla    V^{1}_{3,1}\|^{2}\right)-\frac{1}{4}\left(  \|\nabla   V^{3}_{1,2}\|^{2}+\|\nabla   V^{1}_{2,3}\|^{2}\right).
\end{align*}
Using \eqref{eq:Vij4}, one may verify that the LHS (hence the RHS) 
of the above three equations are zero outside $B$.
Further, using \eqref{eq:Vij3} (note the vectors in $\Pc$ make \eqref{eq:Vij3} applicable here), 
we see that applying $\Ec_{A,b}$ to the RHS of the above three equations gives us zero on $\R^3$. Hence unique continuation for solutions of second order elliptic PDEs implies
that the RHS of the above three equations are zero on $\R^3$. Therefore
\begin{equation}\label{orthogonality1}
    \langle\nabla   V^{2}_{1,2},\nabla   V^{3}_{2,3}\rangle= \langle \nabla   V^{3}_{2,3},\nabla   V^{1}_{3,1}\rangle= \langle\nabla   V^{2}_{1,2},\nabla   V^{1}_{3,1}\rangle=0,
    \qquad \text{on } \R^3.
\end{equation}

Similarly, using \eqref{eq:Vij4}, we have
\[
\| \nabla V^2_{1,2}\|^2 = 1 \qquad \text{outside } B
\]
and, using \eqref{eq:Vij3}, applying $\Ec_{A,b}$ to the above equation results in $0$ on both sides, on $\R^3$, 
hence unique continuation implies the above relation holds on $\R^3$. Repeating this argument, one obtains
\begin{align}\label{orthogonality2}
    \|\nabla   V^{2}_{1,2}\|^2=\|\nabla   V^{3}_{2,3}\|^2=\|\nabla   V^{1}_{3,1}\|^2=1,
    \qquad \text{on } \R^3.
\end{align}
Using \eqref{orthogonality1} and \eqref{orthogonality2} and repeating the argument on \cite[page 325]{ORS26a}, we
conclude that $\Psi$ is a diffeomorphism and $\Psi^*(Eucl) = g$. Further, \eqref{eq:Vij4} shows that
$\Psi=id$ outside $B$.

%%%%%%%%%%%%%%%%%%%%

\begin{appendices}

%%%%%%%%%%%%%
%%%%%%
\section{Proof of Proposition \ref{prop:EIC}}

The proof is similar to the proof of \cite[Proposition 2.3]{ORS26a}, where we have replaced 
geodesics with initial velocity parallel to a 
fixed unit vector $\omega$ by geodesics originating at a fixed point $p$. Our proof is a little simpler.

The geodesics $\gamma_{p,\omega}$ will be called $p$-geodesics. Since $p$ is fixed, instead of writing 
$\Phi_p$ or $\gamma_{p,\omega}$ we write $\Phi$ and $\gamma_\omega$. 
We prove $\Phi$ restricted to $(0,T) \times \Stwo$ is a diffeomorphism through a sequence of steps.
The following observation will be useful in the proof. Given $x \in \R^3$, by \cite[Corollary 6.15]{Lee2018} of the 
Hopf-Rinow theorem, there is a geodesic $\gamma$ from $p$ through $x$ such that
\[
d(p,x) = L(\gamma|_{[p,x]}),
\]
where $\gamma|_{[p,x]}$ is the part of $\gamma$ from $p$ to $x$.
If we reparametrize $\gamma$ so that it has unit speed, $\gamma$ will become a $p$-geodesic. Hence, for every
$x \in \R^3$, there is a $p$-geodesic $\gamma$ such that $d(p,x) = L(\gamma|_{[p,x]})$.

Given $x \in \Bpone$, if $\gamma:[0,1] \to \R^3$ is the straight line path from $p$ to $x$ then
\[
L(\gamma) = \int_0^1 \| \gammadot(s) \| \, ds \leq \sqrt{g_{max}} \int_0^1 |\gammadot(s)| \, ds 
= \sqrt{g_{max}} |x-p| \leq \sqrt{g_{max}} ( |p| + 1) < T.
\]
Hence $d(x,p) < T$ for all $x \in \Bpone$ so
\[
\Bpone \setminus \{p\} \subset D_T.
\]

Let $\Sigma := \{ x \in \R^3 : |x-p| = |p|+1 \}$ denote the $p$-centered sphere of radius $|p|+1$. If a $p$-geodesic reaches 
$\Sigma$, it must cross $\Sigma$ transversally, travel out of $\Bpone$, and never enter $\Bpone$ again, because
all $p$-geodesics originate at $p$ and the parts of the geodesic not in $\Bpone$ are line segments.

\noindent
\underline{We show that every $p$-geodesic crosses $\Sigma$ before time $T$.}

We have already seen that given a point on $\Sigma$, some $p$-geodesic reaches that point by time
$T_1 := \sqrt{g_{max}}(|p|+1) <T$. Further the EIC condition 
guarantees that, given a point on $\Sigma$, there is exactly one $p$-geodesic which reaches that point before time $T$.
Since $\Phi$ is continuous and no $p$-geodesic
can reach $\Sigma$ before time $|p|-1$, the set 
\[
K = \Phi^{-1} (\Sigma) \cap ( [|p|-1, T_1] \times \Stwo)
\]
is compact and $\Phi : K \to \Sigma$ is a continuous bijection. Since $K$ is compact, $\Phi$ is a diffeomorphism. Hence, using the inverse, we have continuous maps $\Omega : \Sigma \to \Stwo$ and $\tau : \Sigma \to \R$ such that
\[
\Phi( \Omega(x), \tau(x) ) = x, \qquad x \in \Sigma.
\]
Now $\Omega$ is injective, because if it was not then there would be two $p$-geodesics reaching the same point on
$\Sigma$ at (possibly different) times less than $T$, which would contradict the EIC. Let $M \subset \Stwo$ be 
the range of $\Omega$; then $\Omega : \Sigma \to M$ is a continuous bijection, hence a homeomorphism since
$\Sigma$ is compact. 
 \begin{figure}[h]
    \centering
\includegraphics[width=0.5\textwidth]{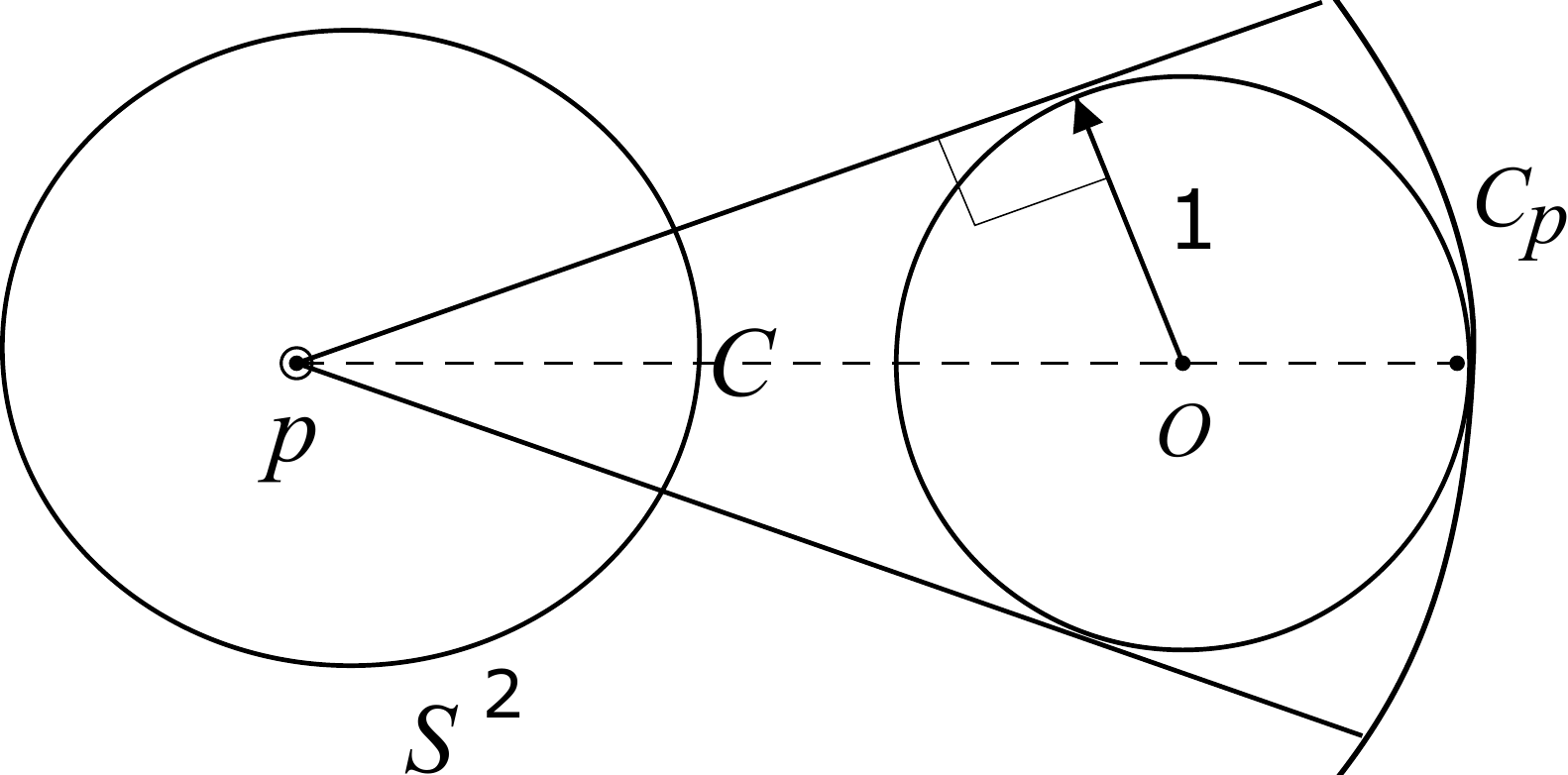}
    \caption{Spherical caps}
    \label{fig:caps}
\end{figure}
%\begin{figure}[h]
%\begin{center}
%\epsfig{file=caps, height=1.5in}
%\end{center}
%\caption{Spherical caps}
%\label{fig:caps}
%\end{figure}
%
Define the spherical caps 
\begin{align*}
C_p & := \{ x \in \Sigma : (x-p) \cdot (-p) \geq (|p|+1) \,\sqrt{|p|^2-1} \},
\\
C & := \{ w \in \Stwo : \omega \cdot (-p ) \geq \sqrt{|p|^2-1} \}.
\end{align*}

of $\Stwo$ and $\Sigma$. Then $\Omega$ maps $\Sigma \setminus C_p$ to $\Stwo \setminus C$ bijectively, hence
the restricted map $\Omega : C_p \to M \cap C$ is a homeomorphism. Therefore they have the same homology, so 
$H_1(M \cap C) = H_1(C_p)$. Now $C_p$ is contractible to a point, hence $H_1(C_p) = (0)$, so $H_1(M \cap C) = (0)$.
Also, one observes that $\pa C \subset M$, and if $M \cap C \neq C$ then, using one of the points in 
$C \setminus M$, one can construct a retract\footnote{A continuous map $r: M \cap C \to \pa C$ such that
$r(\omega) = \omega$ for all $\omega \in \pa C$.} $r : M \cap C \to \pa C$. Then, if $i : \pa C \to M$ is the inclusion map,
we have $r \circ i$ is the identity map on $\pa C$. 
Note $\pa C$ is homeomorphic to $\Sb^1$, so $H_1(\pa C) = \Zb \neq (0)$, 
therefore $r \circ i = id$ implies $H_1( M \cap C)   \neq (0)$ - a contradiction. Therefore $M \cap C = C$, implying every $p$-geodesic crosses $\Sigma$ before time $T$.

\noindent
\underline{Next we show that $\Phi$ is injective on $(0,T) \times \Stwo$.}

Through out we assume we are working with the part of the $p$-geodesics corresponding to the time interval $(0,T)$.
We have to show that no two $p$-geodesics intersect, arriving at that point at possibly different times.  
If two $p$-geodesics intersect, from the 
EIC we know that this point of intersection must be in $B$. We also know that once a $p$-geodesic leaves B, it never 
returns to $B$. So suppose $\gamma_1$ and $\gamma_2$ are two distinct $p$-geodesics, which intersect at some point
$q \in B$; after that they will cross $\Sigma$, say at points $q_1, q_2$ respectively. From the EIC we know that
$q_1, q_2$ are different and $\gamma_i$ is the shortest path from $p$ to $q_i$. WLOG we assume that
\[
L(\gamma_1|_{[p,q]}) \leq L(\gamma_2|_{[p,q]}).
\]
Consider the curve $\sigma$ from $p$ to $q_2$ consisting of $\gamma_1$ from $p$ to $q$ followed by
$\gamma_2$ from $q$ to $q_2$; then
\[
L(\sigma) \leq L(\gamma_2|_{[p,q_2]}).
\]
But $\gamma_2|_{[p,q_2]}$ is the shortest curve from $p$ to $q_2$, hence
\[
L(\sigma) = L(\gamma_2|_{[p,q_2]})
\]
and $\sigma$ is another shortest curve from $p$ to $q_2$. Hence, from \cite[Corollary 26, Chapter 10]{ONeill},
$\sigma$ is a geodesic so a $p$-geodesic. However there is only one $p$-geodesic from $p$ to $q_2$, namely $\gamma_2$, so 
\[
\gamma_1|_{[p,q]} = \gamma_2|_{[p,q]},
\]
hence $\gamma_1=\gamma_2$, a contradiction.

\noindent
\underline{Next we show $\Phi$ restricted to $(0,T) \times \Stwo$ is a diffeomorphism}.

To prove this, it remains to show that $\Phi$ is a local diffeomorphism.
We use the terminology from \cite[Chapter 10]{Lee2018}. If $\Phi'(t_0, \omega_0)$ is singular for some
$(t_0, \omega_0) \in (0,T) \times \Stwo$ then, from \cite[Proposition 10.11]{Lee2018}, $q=\Phi(t_0,\omega_{0})$ is
a conjugate to $p$ along $\gamma_{\omega_0}$. Hence, by \cite[Theorem 10.15]{Lee2018}, $\gamma_{\omega}$
is not the shortest path from $p$ to points on $\gamma_{\omega}$ past $q$. Since every $p$-geodesic crosses
$\Sigma$ and $t_0<T$, there is a $t_1 \in (t_0, T)$ such that $q_1 = \gamma_{\omega_0}(t_1)$ lies outside
$B$. Then $\gamma_{\omega_0}$ is a $p$-geodesic through $q_1$ but is not the shortest curve from $p$ to $q_1$. 
Hence there is another $p$-geodesic from $p$ to $q_1$ of length less than 
$T$. This violates the EIC since $q_1$ is outside $B$.
Hence $\Phi$ restricted to $(0,T) \times \Stwo$ is a local diffeomorphism, hence the injectivity of $\Phi$ implies it is a diffeomorphism.

It remains to prove the three bulleted claims. Since $\Phi$ restricted to $(0,T) \times \Stwo$ is a diffeomorphism, 
we have
a unique smooth map $x \to (r_x, \omega_x)$ from $D_T$ to $(0,T) \times \Stwo$ so that
$x = \gamma_{\omega_x}(r_x)$; further $\gamma_{\omega_x}$ is the shortest curve from $p$ to $x$. Hence
\[
\alpha(x) = d(p,x) = \int_0^{r_x} \| \gammadot_{\omega_x}(t) \| \, dt = r_x,
\]
and $\alpha$ is a smooth function on $D_T$. The second bulleted claim is a consequence of the third bulleted claim 
because $\gamma_{p,\omega}$ is a unit speed geodesic. The third bulleted claim is standard in the theory of the 
construction of solutions of nonlinear first order nonlinear PDEs (see \cite[Section 3.2, Theorem 2]{Evans-pde-book}), 
or one can imitate the proof of the same claim for Lorentzian manifolds in \cite[Lemma A.1]{ORS26b}.

%%%%%%%%%%%%%%
%%%%%%%
\section{Proof of Proposition \ref{prop:forward}}

%%%%%%

The principal symbol of $\mathcal{L}_{A,b}$ is the same as the principal symbol of $\Box_h$ for the 
Lorentzian metric (here $g=A^{-1}$)
\[
h(x,t) := -dt^2+g_{ij}(x)dx_{i}dx_{j}\quad\text{on } \R^3 \times \R,
\]
so the operator $\mathcal{L}_{A,b}$ is {\em normally hyperbolic} with $\pa_t$ as the global future directed vector 
field -  
see \cite[Section 1.5]{Bar}. Since $t= c$ is a smooth spacelike Cauchy surface for this Lorentzian metric, 
$\mathcal{L}_{A,b}$ is also {\em globally hyperbolic} - see \cite[Section 2]{ORS26b}. 
By \cite[Corollary 3.4.3]{Bar}, there 
exist advanced and retarded Green's operator $G_{+}, G_{-}$ associated with $\mathcal{L}_{A,b}$
- see\cite[Definition 3.4.1]{Bar}. Therefore \cite[Theorem 3.8]{Bar15} guarantees that 
$G_{+}, G_{-}$ have unique linear extensions 
$\overline{G_{+}}$, $\overline{G_{-}}$, respectively. Thus the hypothesis of \cite[Lemma 4.1]{Bar15} is satisfied. 

The geodesics associated with the Lorentzian metric $h$ may be parametrized by $t$ and every geodesic of $h$ is of the form $t \to (\gamma(t), t))$ for some geodesic $\gamma$ of $g$. Further, these geodesics are defined 
for all $t$ since there is a real number $g_{max}$ such that
\[
g(x)(v,v) \leq g_{max}|v|^2, \qquad \forall (x,v) \in T(\R^3).
\]

The support of the RHS of \eqref{eq:Upde} is
\[
M := \supp H(t) \, \delta(x-p) = \{ (p,t) : t \ge 0 \},
\]
 and $M$ is past-compact because of the finite speed of propagation of the Riemannian geodesics for $g$. 
 Hence \cite[Lemma 4.1, Corollary 4.2]{Bar15} implies that the IVP \eqref{eq:Upde}, \eqref{eq:Upic} has a unique distributional solution and
 \[
 \supp U_p \subset J_+(M).
\]
If $(q,s) \in J_+(M)$ then there is a $t_0 \in [0,s]$ and a time parametrized continuous piecewise smooth 
curve $\gamma : [t_0,s] \to \R^3$ such that
$\gamma(t_0)=p, \gamma(s)=q$ and $\|\gammadot(t)\| \leq 1$ for all $t \in [t_0,s]$. Hence
\[
L(\gamma) = \int_{t_0}^s \|\gammadot(t)\| \, dt \leq \int_{t_0}^s \, dt = s-t_0,
\]
so
\[
\alpha_p(q) = d(p,q) \leq L(\gamma) \leq s- t_0 \leq s,
\]
implying
\[
\supp U_p \subset \{(x,t) : t \geq \alpha_p(x) \}.
\]
 
 Noting that 
 \[
 \WF(U_p) \supset \WF ( \Lc_{A,b} U_p ) = \WF( H(t) \delta(x-p) ) = \Zc,
 \]
 and using \cite[Chapter VI, Theorem 2.1]{taylor1981}, we see that $\WF(U_p)$ is the union of $\Zc$
 with the forward in time null bicharacteristic flow out of $\Zc$. The only null vectors in $\Zc$ are
 $(p,0; \xi, \tau)$ with $|\xi|^2 = \tau^2$, so the forward in time null bicharacteristic flow out of $\Zc$
 consists of the $p$-null bicharacteristics. Hence
 \[
 \WF(U_p) = \Lambda_p^* \, \cup \, \Zc.
 \]

 For $t < |p|-1$, the distribution ${\frac{H(t-|x-p|)}{|x-p|} }$
 is supported in $ B_p(|p|-1) \times [0, |p|-1)$, and $\Lc_{A,b} = \Box$ on $B_p(|p|-1) \times \R$. We also know that
 \[
 \Box \left (  \frac{H(t-|x-p|)}{|x-p|}  \right ) = 4 \pi H(t) \, \delta(x-p), \qquad \text{on } \R^3 \times \R,
 \]
 hence, by a standard uniqueness argument for IVP for hyperbolic PDEs, we have
 \[
 U_p(x,t) =  \frac{H(t-|x-p|)}{|x-p|} , \qquad \text{on } \R^3 \times (-\infty, |p|-1).
 \]

Let \(\phi (x) \in C_c^\infty (\R^3) \) with $\phi=1$ in a neighborhood of $p$ and $\supp \chi \cap \Bbar = \emptyset$;
define
\[
V(x,t) :=\phi(x)\frac{H(t-|x-p|)}{|x-p|}, \qquad \text{on } \mathbb{R}^3\times \R.
\]

One may verify that $V \in L^2_{loc}(\R^3 \times \R)$ and
\begin{equation}
\Lc_{A,b} V(x,t)=4 \pi H(t)\delta(x-p)+h(x)\delta(t-|x-p|)+k(x)H(t-|x-p|) 
\end{equation}
for some $h, k \in C_c^\infty(\R^3)$ which are zero in a neighborhood of $p$. Hence
\[
W = V - U_p
\]
is the solution of the IVP
\begin{subequations}
\begin{align}
\Lc_{A,b} W =F, & \qquad \text{on } \R^3 \times \R
\label{eq:Wde}
\\
W=0, & \qquad \text{on } \R^3 \times (-\infty, 0),
\label{eq:Wic}
\end{align}
\end{subequations}
where
\[
F(x,t) := h(x)\delta(t-|x-p|)+k(x)H(t-|x-p|), \qquad \text{on } \R^3 \times \R.
\]
Now $F$ is supported in $t \geq 0$. Since $h(x), k(x)$ are compactly supported and zero in a neighborhood of $p$ and
$h(x) \delta(t-|x-p|) = \pa_t ( h(x) H(t-|x-p|))$, we see that $F \in H^{-1}_{loc} (\R^3 \times \R)$.
Hence, by \cite[Chapter VI, Proposition 1.10, Theorem 2.1]{taylor1981},
$W$ is one degree smoother than $F$ in the null directions and 
two degrees smoother than $F$ in the non-null directions in $T^*(\R^3 \times \R)$. 
Therefore $W \in L^2_{loc}(\R^3 \times \R)$, implying $U_p = V - W \in L^2_{loc}(\R^3 \times \R)$.
%%%%%%%%%%%%%%%%%%%%%%%%%%%%%%%%%%%%%%%%%
%%%%%%%%%%%%%%%
%%%%%%%%%%%%%%

%%%%%%%%%%%%%%
\section{Energy estimates}

%%%%%%%%%
%
 \begin{figure}[h]
    \centering
\includegraphics[width=0.33\textwidth]{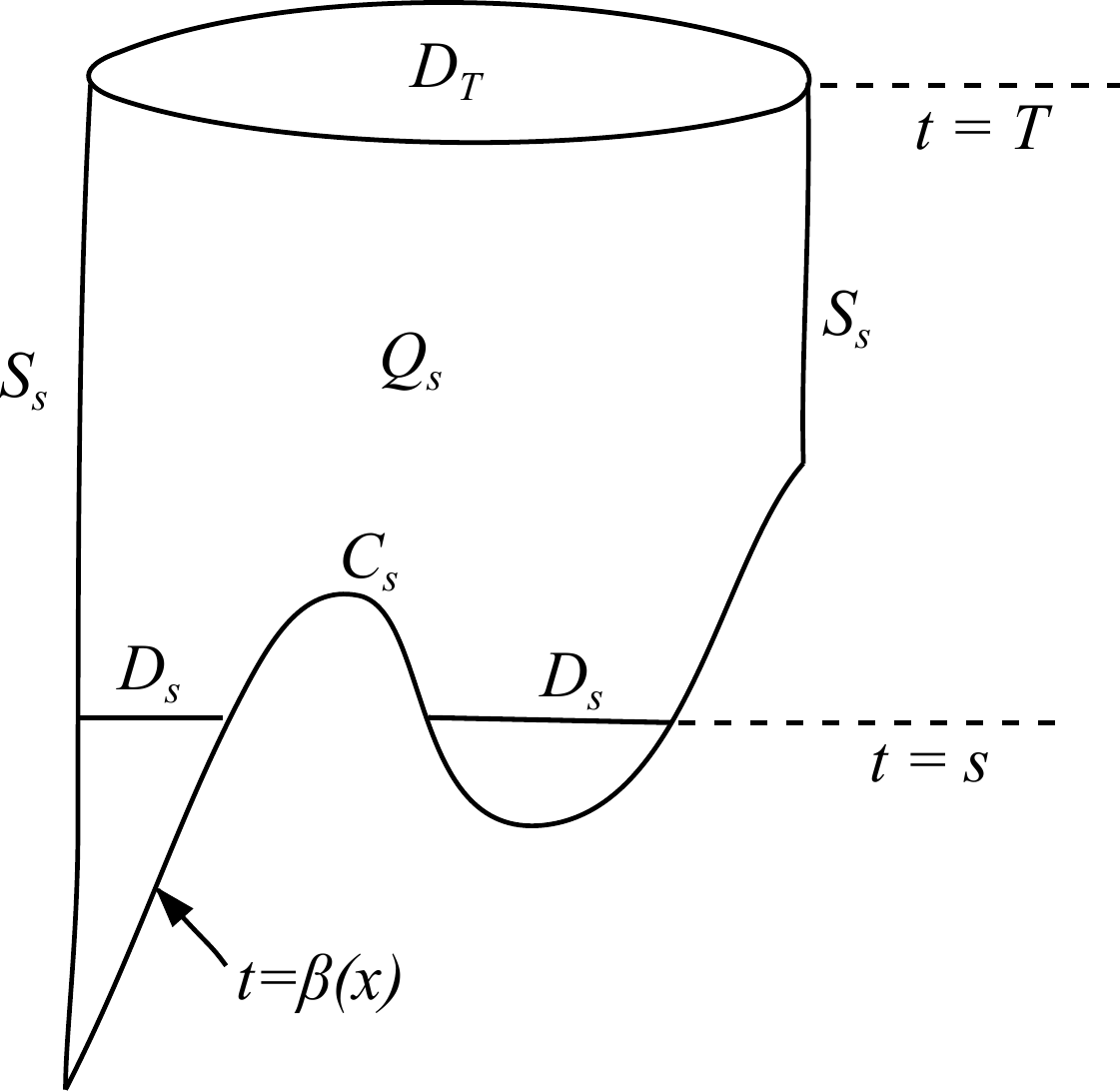}
   \caption{The regions and surfaces for Proposition \ref{prop:energy}}
\label{fig:energy}
\end{figure}
%\begin{figure}[ht]
%\begin{center}
%\epsfig{file=energy, height=2in}
%\end{center}
%\caption{The regions and surfaces for Proposition \ref{prop:energy}}
%\label{fig:energy}
%\end{figure}
%
Suppose $\beta(x)$ is a smooth function on $\Bbar$ and $T$ a real number with
\[
T > \beta_{\max} := \max_{x \in \Bbar} \beta(x).
\]
Define the region
\[
Q :=\{(x,t): x \in \Bbar, \beta(x)\leq t \leq T\} \subset \Bbar \times\mathbb{R}
\]
and, for $s \in \R$, the surfaces (see Figure \ref{fig:energy})
\begin{gather*}
Q_s:= Q \cap \{t \geq s \}, \qquad D_s:= Q \cap \{ t=s \}
\\
 S_s := \pa Q_s \cap (\pa B \times \R), \qquad C_s:= \pa Q_s \cap \{ (x, \beta(x)) : x \in \Bbar \}.
\end{gather*}
Note that 
\[
\pa Q_s = D_T \cup D_s \cup S_s \cup C_s
\]
with $D_T$ the top flat part, $D_s$ is the bottom flat part,  $S_s$ the lateral cylindrical part,
and $C_s$ the part lying on the characteristic surface $t=\beta(x)$;  $C_s$ may be empty for some $s$.

%%%%%%%%%%%%%%%%%%
\begin{proposition}\label{prop:energy}
Suppose $\Lc_{A,b}$ is an admissible operator, $\beta(x)$ a smooth function on $\Bbar$ such that 
$\|\nabla_g \beta\|^2 =1 $ on $\Bbar$ for $g=A^{-1}$, and $T > \beta_{max} := \max_{x \in \Bbar} \beta(x)$.
If $w(x,t)$ is a smooth function on $Q$ such that
\begin{subequations}
\begin{align}
\Lc_{A,b} w =0, & \qquad \text{on } Q,
\label{eq:wde}
\\
w(x, T) =0, ~ w_t (x, T)=0, & \qquad \text{for } x \in \Bbar,
\label{eq:wic}
\\
w = 0, & \qquad \text{on } S_{\beta_{min}}
\label{eq:wbc}
\end{align}
\end{subequations}
then $w=0$ on $Q$. Here $\beta_{min} = \min_{x \in \Bbar} \beta(x)$.
\end{proposition}

%%%%%%%%%%
%%%%%%%
\begin{proof}
The proof is a standard application of energy estimates on bounded regions for hyperbolic PDEs, 
except that one part of the boundary is a characteristic surface. Estimates for such regions are also well studied, and
\cite[Proposition B.1]{ORS26a} is similar to our proposition except we have the more general operator $\Lc_{A,b}$.
This general operator requires the use of Gronwall's inequality. 

Using standard manipulations, we have the identity
\begin{align}
0 & = 2 (\Lc_{A,b} w ) \, w_t = \pa_t F + \nabla \cdot G + R
\label{eq:Lwwt}
\end{align}
with
\begin{gather*}
F  := w_t^2 + (\nabla w)^T A \nabla w + w^2, \qquad G  := - 2 w_t A \nabla w,
\\
R  := 2 (w_t \, b \cdot \nabla w -  w_t \, \pa_i (a_{ij}) \pa_j w +  w \,w_t).
\end{gather*}

Define the energy at time $t$ as
\[
E(t) := \int_{D_t} (w_t^2 + (\nabla w)^T A \nabla w + w^2)(x,t) \, dx.
\]
Integrating \eqref{eq:Lwwt} over $Q_s$ and using the divergence theorem we have
\begin{align}
E(T) - E(s) - & 2 \int_{S_s} w_t (x^T A \nabla w) \, dS 
 - \int_{C_s} ( F -  \nabla \beta \cdot G) ( 1 + |\nabla \beta|^2)^{-1/2} \, dS
\nonumber
\\
& + \int_{Q_s} R = 0.
\label{eq:energy1}
\end{align}
From the hypothesis, the integrals on $D_T$ and $S_s$ are zero (note $w_t=0$ on $S_s$); further using the positive definiteness of $A$, we have
\[
|R|   \cleq  w_t^2 + (\nabla w)^T A \nabla w + w^2, \qquad \text{on } Q.
\]
Hence \eqref{eq:energy1} implies
\begin{align}
E(s) + \int_{C_s} ( F -  \nabla \beta \cdot G) ( 1 + |\nabla \beta|^2)^{-1/2} \, dS 
\cleq \int_s^T E(t) \, dt,
\qquad s \leq T
\end{align}
If we can show that the integrand in the integral on $C_s$ is non-negative, then
\[
E(s) \cleq \int_s^T E(t) \, dt, \qquad s \leq T,
\]
which, by Gronwall's inequality, implies $E(s)=0$ for all $s \leq T$. Hence $w=0$ on $Q$.

It remains to show that $F -  \nabla \beta \cdot G \geq 0$ on $C_s$. Since $\|\nabla_g \beta \|^2=1$,
and $g=A^{-1}$ so $\nabla_g f = g^{-1} \nabla f = A \nabla f$, we have
\begin{align*}
F -  \nabla \beta \cdot G & = w_t^2 + (\nabla w)^T A \nabla w + w^2 + 2 w_t \, (\nabla \beta)^T A \nabla w
\\
& = \|\nabla_g \beta\|^2 w_t^2 + \|\nabla_g w\|^2 + 2 w_t \, g( \nabla_g \beta, \nabla_g w) + w^2
\\
& = \| w_t \nabla_g \beta + \nabla_g w\|^2 + w^2 \geq 0.
\end{align*}
\end{proof}
\end{appendices}

\bibliographystyle{unsrtnat}
\bibliography{references}

@misc{AK25,
      title={Inverse spectral problems with sparse data and applications to passive imaging on manifolds}, 
      author={Ali Feizmohammadi and Katya Krupchyk},
      year={2025},
      eprint={2507.22723},
      archivePrefix={arXiv},
      primaryClass={math.AP},
      url={https://arxiv.org/abs/2507.22723}, 
}

@misc{ORS26d,
      title={Semiglobal uniqueness for the {L}orentzian {C}alder\'on problem}, 
      author={Lauri Oksanen and Rakesh and Mikko Salo},
      year={2026},
      eprint={2608.13116},
      archivePrefix={arXiv},
      primaryClass={math.AP},
      url={https://arxiv.org/abs/2608.13116}, 
}

@misc{ORS26c,
      title={Fixed angle inverse scattering with non-constant velocity}, 
      author={Lauri Oksanen and Rakesh and Mikko Salo},
      year={2026},
      eprint={2608.13670},
      archivePrefix={arXiv},
      primaryClass={math.AP},
      url={https://arxiv.org/abs/2608.13670}, 
}

@article {ORS26b,
    AUTHOR = {Oksanen, Lauri and Rakesh and Salo, Mikko},
     TITLE = {Rigidity in the {L}orentzian {C}alder\'on problem with
              formally determined data},
   JOURNAL = {Commun. Am. Math. Soc.},
  FJOURNAL = {Communications of the American Mathematical Society},
    VOLUME = {6},
      YEAR = {2026},
     PAGES = {749--792},
      ISSN = {2692-3688},
   MRCLASS = {35R30},
  MRNUMBER = {5092743},
       DOI = {10.1090/cams/73},
       URL = {https://doi.org/10.1090/cams/73},
}

@article{ORS26a,
author = {Oksanen, Lauri and Rakesh and Salo, Mikko},
title = {Rigidity for Fixed Angle Inverse Scattering for {R}iemannian Metrics},
journal = {SIAM Journal on Mathematical Analysis},
volume = {58},
number = {1},
pages = {308-334},
year = {2026},
doi = {10.1137/24M1702623},
URL = {https://doi.org/10.1137/24M1702623},
eprint = {https://doi.org/10.1137/24M1702623
}
}

@article {KRS21,
    AUTHOR = {Krishnan, Venkateswaran P. and Rakesh, Rakesh and Senapati,
              Soumen},
     TITLE = {Stability for a formally determined inverse problem for a
              hyperbolic {PDE} with space and time dependent coefficients},
   JOURNAL = {SIAM J. Math. Anal.},
  FJOURNAL = {SIAM Journal on Mathematical Analysis},
    VOLUME = {53},
      YEAR = {2021},
    NUMBER = {6},
     PAGES = {6822--6846},
      ISSN = {0036-1410,1095-7154},
   MRCLASS = {35R30 (35L05 35L20)},
  MRNUMBER = {4347324},
       DOI = {10.1137/21M1400596},
       URL = {https://doi.org/10.1137/21M1400596},
}

@article {IY2001,
    AUTHOR = {Imanuvilov, Oleg Yu. and Yamamoto, Masahiro},
     TITLE = {Global uniqueness and stability in determining coefficients of
              wave equations},
   JOURNAL = {Comm. Partial Differential Equations},
  FJOURNAL = {Communications in Partial Differential Equations},
    VOLUME = {26},
      YEAR = {2001},
    NUMBER = {7-8},
     PAGES = {1409--1425},
      ISSN = {0360-5302,1532-4133},
   MRCLASS = {35R30 (35R25)},
  MRNUMBER = {1855284},
MRREVIEWER = {Peter\ G.\ Danilaev},
       DOI = {10.1081/PDE-100106139},
       URL = {https://doi.org/10.1081/PDE-100106139},
}

@book{Lee2018,
	address = {New York, NY},
	title = {Introduction to {Riemannian} manifolds},
	isbn = {978-3-319-91754-2},
	publisher = {Springer Berlin Heidelberg},
	author = {Lee, John},
	year = {2018},
}

@article {Bukhgeum81,
    AUTHOR = {Bukhgeim, A. L. and Klibanov, M. V.},
     TITLE = {Uniqueness in the large of a class of multidimensional inverse
              problems},
   JOURNAL = {Dokl. Akad. Nauk SSSR},
  FJOURNAL = {Doklady Akademii Nauk SSSR},
    VOLUME = {260},
      YEAR = {1981},
    NUMBER = {2},
     PAGES = {269--272},
      ISSN = {0002-3264},
   MRCLASS = {35R30 (35K10)},
  MRNUMBER = {630135},
MRREVIEWER = {R.\ G.\ Mukhometov},
}

@article {Salo20A,
    AUTHOR = {Rakesh and Salo, Mikko},
     TITLE = {The fixed angle scattering problem and wave equation inverse
              problems with two measurements},
   JOURNAL = {Inverse Problems},
  FJOURNAL = {Inverse Problems. An International Journal on the Theory and
              Practice of Inverse Problems, Inverse Methods and Computerized
              Inversion of Data},
    VOLUME = {36},
      YEAR = {2020},
    NUMBER = {3},
     PAGES = {035005, 42},
      ISSN = {0266-5611,1361-6420},
   MRCLASS = {35R30 (35L15 35P25 65J22)},
  MRNUMBER = {4068234},
MRREVIEWER = {Isaac\ Harris},
       DOI = {10.1088/1361-6420/ab23a2},
       URL = {https://doi.org/10.1088/1361-6420/ab23a2},
}

@article {Salo20,
    AUTHOR = {Rakesh and Salo, Mikko},
     TITLE = {Fixed angle inverse scattering for almost symmetric or
              controlled perturbations},
   JOURNAL = {SIAM J. Math. Anal.},
  FJOURNAL = {SIAM Journal on Mathematical Analysis},
    VOLUME = {52},
      YEAR = {2020},
    NUMBER = {6},
     PAGES = {5467--5499},
      ISSN = {0036-1410,1095-7154},
   MRCLASS = {35R30 (81U40)},
  MRNUMBER = {4170189},
MRREVIEWER = {Mansur\ I.\ Ismailov},
       DOI = {10.1137/20M1319309},
       URL = {https://doi.org/10.1137/20M1319309},
}

@article {Salo21,
    AUTHOR = {Mero\~no, Crist\'obal J. and Potenciano-Machado, Leyter and
              Salo, Mikko},
     TITLE = {The fixed angle scattering problem with a first-order
              perturbation},
   JOURNAL = {Ann. Henri Poincar\'e},
  FJOURNAL = {Annales Henri Poincar\'e. A Journal of Theoretical and
              Mathematical Physics},
    VOLUME = {22},
      YEAR = {2021},
    NUMBER = {11},
     PAGES = {3699--3746},
      ISSN = {1424-0637,1424-0661},
   MRCLASS = {35R30 (81U40)},
  MRNUMBER = {4325875},
       DOI = {10.1007/s00023-021-01081-w},
       URL = {https://doi.org/10.1007/s00023-021-01081-w},
}

@article {Salo22,
    AUTHOR = {Ma, Shiqi and Salo, Mikko},
     TITLE = {Fixed angle inverse scattering in the presence of a
              {R}iemannian metric},
   JOURNAL = {J. Inverse Ill-Posed Probl.},
  FJOURNAL = {Journal of Inverse and Ill-Posed Problems},
    VOLUME = {30},
      YEAR = {2022},
    NUMBER = {4},
     PAGES = {495--520},
      ISSN = {0928-0219,1569-3945},
   MRCLASS = {35Q60 (31B10 35J05 35R30 78A40)},
  MRNUMBER = {4460195},
       DOI = {10.1515/jiip-2020-0119},
       URL = {https://doi.org/10.1515/jiip-2020-0119},
}

@article {Salo23,
    AUTHOR = {Ma, Shiqi and Potenciano-Machado, Leyter and Salo, Mikko},
     TITLE = {Fixed angle inverse scattering for sound speeds close to
              constant},
   JOURNAL = {SIAM J. Math. Anal.},
  FJOURNAL = {SIAM Journal on Mathematical Analysis},
    VOLUME = {55},
      YEAR = {2023},
    NUMBER = {4},
     PAGES = {3420--3456},
      ISSN = {0036-1410,1095-7154},
   MRCLASS = {35R30 (31B10 35J05 35Q60 78A40)},
  MRNUMBER = {4624351},
       DOI = {10.1137/22M147640X},
       URL = {https://doi.org/10.1137/22M147640X},
}

@book{ONeill,
  title={Semi-Riemannian Geometry With Applications to Relativity},
  author={O'Neill, B.},
  isbn={9780080570570},
  series={Pure and Applied Mathematics},
  url={https://books.google.com/books?id=CGk1eRSjFIIC},
  year={1983},
  publisher={Academic Press}
}

@article {Bar15,
    AUTHOR = {B\"ar, Christian},
     TITLE = {Green-hyperbolic operators on globally hyperbolic spacetimes},
   JOURNAL = {Comm. Math. Phys.},
  FJOURNAL = {Communications in Mathematical Physics},
    VOLUME = {333},
      YEAR = {2015},
    NUMBER = {3},
     PAGES = {1585--1615},
      ISSN = {0010-3616,1432-0916},
   MRCLASS = {35L30 (35L15 35L45 35R01 58J45)},
  MRNUMBER = {3302643},
MRREVIEWER = {Herbert\ Koch},
       DOI = {10.1007/s00220-014-2097-7},
       URL = {https://doi.org/10.1007/s00220-014-2097-7},
}

@book {Bar,
    AUTHOR = {B\"ar, Christian and Ginoux, Nicolas and Pf\"affle, Frank},
     TITLE = {Wave equations on {L}orentzian manifolds and quantization},
    SERIES = {ESI Lectures in Mathematics and Physics},
 PUBLISHER = {European Mathematical Society (EMS), Z\"urich},
      YEAR = {2007},
     PAGES = {viii+194},
      ISBN = {978-3-03719-037-1},
   MRCLASS = {58J45 (35L05 35L15 46L60 81R15 81T05 81T20)},
  MRNUMBER = {2298021},
MRREVIEWER = {Rainer\ Verch},
       DOI = {10.4171/037},
       URL = {https://doi.org/10.4171/037},
}

@article {KRS23,
    AUTHOR = {Krishnan, Venkateswaran P. and Rakesh and Senapati, Soumen},
     TITLE = {Point sources and stability for an inverse problem for a
              hyperbolic {PDE} with space and time dependent coefficients},
   JOURNAL = {J. Differential Equations},
  FJOURNAL = {Journal of Differential Equations},
    VOLUME = {342},
      YEAR = {2023},
     PAGES = {622--665},
      ISSN = {0022-0396,1090-2732},
   MRCLASS = {35R30 (35K20 35L05 35L20)},
  MRNUMBER = {4502222},
       DOI = {10.1016/j.jde.2022.10.025},
       URL = {https://doi.org/10.1016/j.jde.2022.10.025},
}

@article {R,
    AUTHOR = {Rubin, Boris},
     TITLE = {Inversion formulae for the spherical mean in odd dimensions
              and the {E}uler-{P}oisson-{D}arboux equation},
   JOURNAL = {Inverse Problems},
  FJOURNAL = {Inverse Problems. An International Journal on the Theory and
              Practice of Inverse Problems, Inverse Methods and Computerized
              Inversion of Data},
    VOLUME = {24},
      YEAR = {2008},
    NUMBER = {2},
     PAGES = {025021, 10},
      ISSN = {0266-5611},
   MRCLASS = {44A12 (26A33 35Q15 35R30)},
  MRNUMBER = {2408558},
MRREVIEWER = {Dmitry G. Shepelsky},
       DOI = {10.1088/0266-5611/24/2/025021},
       URL = {https://doi.org/10.1088/0266-5611/24/2/025021},
}

@article {K,
    AUTHOR = {Kunyansky, Leonid A.},
     TITLE = {Explicit inversion formulae for the spherical mean {R}adon
              transform},
   JOURNAL = {Inverse Problems},
  FJOURNAL = {Inverse Problems. An International Journal on the Theory and
              Practice of Inverse Problems, Inverse Methods and Computerized
              Inversion of Data},
    VOLUME = {23},
      YEAR = {2007},
    NUMBER = {1},
     PAGES = {373--383},
      ISSN = {0266-5611},
   MRCLASS = {44A12 (65R10)},
  MRNUMBER = {2302980},
MRREVIEWER = {Fritz Keinert},
       DOI = {10.1088/0266-5611/23/1/021},
       URL = {https://doi.org/10.1088/0266-5611/23/1/021},
}

@book {H1,
    AUTHOR = {H{\"o}rmander, Lars},
     TITLE = {The analysis of linear partial differential operators. {I}},
    SERIES = {Classics in Mathematics},
      NOTE = {Distribution theory and Fourier analysis,
              Reprint of the second (1990) edition [Springer, Berlin;
              MR1065993 (91m:35001a)]},
 PUBLISHER = {Springer-Verlag},
   ADDRESS = {Berlin},
      YEAR = {2003},
     PAGES = {x+440},
      ISBN = {3-540-00662-1},
   MRCLASS = {35-02},
  MRNUMBER = {MR1996773},
}

@Book{taylor1981,
  author	= {Taylor, Michael E.},
  title		= {Pseudodifferential operators},
  fseries	= {Princeton Mathematical Series},
  series	= {Princeton Math. Ser.},
  volume	= {34},
  year		= {1981},
  publisher	= {Princeton University Press, Princeton, NJ},
  language	= {English},
  zbmath	= {3709065},
  zbl		= {0453.47026}
}

@book {Evans-pde-book,
    AUTHOR = {Evans, Lawrence C.},
     TITLE = {Partial differential equations},
    SERIES = {Graduate Studies in Mathematics},
    VOLUME = {19},
 PUBLISHER = {American Mathematical Society, Providence, RI},
      YEAR = {1998},
     PAGES = {xviii+662},
      ISBN = {0-8218-0772-2},
   MRCLASS = {35-01},
  MRNUMBER = {1625845},
MRREVIEWER = {Luigi Rodino},
}

 \end{document}